\documentclass[english,twoside,a4paper,reqno,11pt]{amsart}
\usepackage{amsfonts, amsbsy, amsmath, amsthm, amssymb, latexsym, verbatim, enumerate}
\usepackage{mathrsfs}
\usepackage[top=30mm,right=30mm,bottom=30mm,left=30mm]{geometry}

\usepackage{bm}
\usepackage[pdftex]{color,graphicx}

\makeatletter
\newcommand{\imod}[1]{\allowbreak\mkern4mu({\operator@font mod}\,\,#1)}
\makeatother

\usepackage[T1]{fontenc}
\usepackage[latin9]{inputenc}
\usepackage{textcomp}
\usepackage{amsmath}
\usepackage{amsthm}
\usepackage{amssymb}

\makeatletter

\theoremstyle{plain}
\newtheorem{thm}{\protect\theoremname}[section]
  \theoremstyle{definition}
  \newtheorem{defn}[thm]{\protect\definitionname}
  \theoremstyle{plain}
  \newtheorem{cor}[thm]{\protect\corollaryname}
  \theoremstyle{plain}
  \newtheorem{lem}[thm]{\protect\lemmaname}
  \theoremstyle{plain}
  \newtheorem{prop}[thm]{\protect\propositionname}
  \theoremstyle{remark}
  \newtheorem*{rem*}{\protect\remarkname}

\usepackage{youngtab} 
\usepackage{ytableau} 
\usepackage{mathrsfs}
\usepackage{dsfont}
\usepackage{pdflscape}
\date{}
\usepackage{changepage}

\usepackage{amsthm}

\theoremstyle{plain}
\newtheorem{mainthm}{Theorem}

\makeatother

\usepackage{babel}
  \providecommand{\corollaryname}{Corollary}
  \providecommand{\definitionname}{Definition}
  \providecommand{\lemmaname}{Lemma}
  \providecommand{\propositionname}{Proposition}
  \providecommand{\remarkname}{Remark}
\providecommand{\theoremname}{Theorem}

\begin{document}

\author{Alexander J. Malcolm}

\address{Alexander J. Malcolm, Department of Mathematics,
    Imperial College, London SW7 2BZ, UK}
\email{alexander.malcolm09@imperial.ac.uk}

\title{The Involution Width of Finite Simple Groups}


\begin{abstract}
 For a finite group generated by involutions, the involution
width is defined to be the minimal $k\in\mathbb{N}$ such that any group element
can be written as a product of at most $k$ involutions. We show that
the involution width of every non-abelian finite simple group is at
most $4$. This result is sharp, as there are families with involution
width precisely 4.
\end{abstract}

\date{\today}
\maketitle

\section{Introduction}

 Let $G$ be a finite group generated by involutions. The involution
width, denoted iw$(G)$, is defined to be the minimal $k\in\mathbb{N}$
such that any element of $G$ can be written as a product of at most
$k$ involutions. It is well known that all non-abelian finite simple
groups are generated by their involutions. Furthermore, it was proved
by Liebeck and Shalev (\cite{Liebeck and Shalev}, 1.4) that there
exists an absolute constant $N$ that bounds the involution width
of all finite simple groups. The purpose of this paper is to obtain
the minimal value for $N$.

The involution width was first considered in the case of special linear groups
by Wonenburger \cite{wonenburger}, and later by Gustafson et. al.
\cite{Gustavson et al} and Kn$\ddot{\text{u}}$ppel and Nielsen \cite{Knuppel and Knielsen}.
Further work has been completed on orthogonal groups (Kn$\ddot{\text{u}}$ppel
and Thomsen, \cite{Knuppel and Thomsen 98}) and the exceptional group
$F_{4}(K)$ (Austin, \cite{Austin}). 

More recently, various problems have been resolved involving the related
notions of real or strongly real groups. An element $x\in G$ is real
if $x^{g}=x^{-1}$ for some $g\in G$. Furthermore $x$ is strongly
real if $g$ can be taken to be an involution. We call the group $G$
(strongly) real if all of its elements are (strongly) real. It follows
easily that $x\in G$ is strongly real if and only if $x$ is a product
of 2 involutions and so the strongly real groups are precisely those
of involution width 2. The classification of strongly real, finite
simple groups was completed in 2010 after work by a number of authors
(see  \cite{gow strongly real symplectic,kolesnikov,Johanna Ramo,Ibrahim Suleiman 2008}
and Theorem \ref{Thm: Strongly real G} below). The work of this paper
addresses the remaining finite simple groups and is summarised by
the following theorem.

\begin{mainthm}\label{mainthm} Every non-abelian finite simple group has involution width at most 4. \end{mainthm}

Note that the upper bound 4 is sharp, as certain families (for example
$PSL_{n}(q)$ such that $n,q\geq6$ and gcd$(n,q-1)=1$) do have involution
width 4 (see Theorem \ref{thm:sl 4 ref}).

The involution width is one of a number of width questions that have
been considered about simple groups in recent literature. For example,
\cite{Ore Conjecture paper} settles the longstanding conjecture of
Ore that the commutator width of any finite non-abelian simple group
$G$ is exactly 1. Also, $G$ is generated by its set of squares and
the width in this case is 2 \cite{lost squares}. More generally,
given any two non-trivial words $w_{1},w_{2}$, if $G$ is of large
enough order then $w_{1}(G)w_{2}(G)\supseteq G\backslash\{1\}$ \cite{Guralnick and Tiep}.

Here are some remarks on the proof of Theorem~\ref{mainthm}. For
alternating groups we find the involution width directly by studying
the disjoint cycle decomposition of elements. For groups $G$ of Lie
type we adopt a different approach: we aim to find particular regular
semisimple elements $x,y\in G$ such that $x$ and $y$ are strongly
real and $G\backslash\{1\}\subseteq x^{G}y^{G}$. It then follows
that every element is a product of at most 4 involutions.To do this
we make extensive use of the character theory of finite groups of
Lie type, building on methods first seen in \cite{Malle and Saxl and Weigel}
and \cite{Guralnick and Malle}. Substantial difficulties are faced
in the case of unitary groups, where we develop the theory of minimal
degree characters using dual pairs (Sec. \ref{sec:Unitary-Groups}).
Similarly problematic are a number of exceptional groups of Lie type
and in these instances we use an inductive approach, restricting to
subgroups of $G$ for which the involution width is known (Sec. \ref{sec:Simple-groups-of exceptional type}).

Naturally the involution width problem can be generalised to elements
of order $p$, for any fixed prime $p$. This has been resolved in
the case of alternating groups and work on the simple groups of Lie
type will be forthcoming.

\textit{Acknowledgements: }This paper forms part of work towards a
PhD degree under the supervision of Martin Liebeck at Imperial College
London, and the author wishes to thank him for his direction and support
throughout. The author is also very thankful to Pham Huu Tiep for
helpful advice on the character theory of finite simple groups, and
EPSRC for their financial support.

\section{Alternating and Sporadic Groups}

\subsection{Alternating Groups}

We first consider the simple alternating groups $A_{m}$, $m\geq5$.
Conjugacy in $A_{m}$ is easily understood: classes in the full symmetric
group $S_{m}$ are indexed by cycle type. These types correspond to
partitions of $m$ and an $S_{m}$-class splits in $A_{m}$ if and
only if the partition consists of distinct odd integers. We can find
the involution width directly by studying the permutations and, unlike
later work on groups of Lie type, require no representation theory.\\
For reference, we first note the classification of strongly real finite
simple groups.
\begin{thm}
\label{Thm: Strongly real G}(\cite{gow strongly real symplectic,kolesnikov,Johanna Ramo,Ibrahim Suleiman 2008}
). A finite non-abelian simple group is strongly real if and only
if it is one of the following
\end{thm}
\begin{enumerate}
\item $PSp_{2n}(q)$ where $q\not\equiv3$ (mod $4$) and $n\geq1$;
\item $P\Omega_{2n+1}(q)$ where $q\equiv1$ (mod $4$) and $n\geq3$;
\item $P\Omega_{9}(q)$ where $q\equiv3$ (mod $4$);
\item $P\Omega_{4n}^{+}(q)$ where $q\not\equiv3$ (mod $4$) and $n\geq3$;
\item $P\Omega_{4n}^{-}(q)$ where $n\geq2$;
\item $P\Omega_{8}^{+}(q)$ or $^{3}D_{4}(q)$;
\item $A_{5},$ $A_{6}$, $A_{10},$ $A_{14}$, $J_{1}$, $J_{2}$.
\end{enumerate}
Evidently, all but four of the alternating groups are not strongly
real. We find that products of three involutions are needed in general.
\begin{defn}
Let $G$ be a finite group generated by involutions and let $g\in G$.
The involution width of $g$, denoted iw$(g)$, is the minimal $k\in\mathbb{N}$
such that $g$ can be written as a product of $k$ involutions.
\end{defn}
\begin{thm}
The involution width of $A_{m}$ ($m\geq5$) is at most $3$.\label{An width at most 3}
\end{thm}
\begin{proof}
Let $g\in A_{m}$ and write $g$ as a product of disjoint cycles $g=c_{1}\dots c_{k}$
where $c_{i}\in S_{n}$. Recall that $n_{0}+n_{2}\equiv0$ mod $2$,
where $n_{j}$ denotes the number of cycles of $g$ of length congruent
to $j$ mod $4$. Denote the length of a cycle by $|c_{i}|$. We consider
single cycles and also pairs from this decomposition and show how
they can be written as a product of at most 3 involutions.\\
First consider $c\in\{c_{1},\dots,c_{k}\}$ such that $|c|=n\equiv1\text{ mod }4$.
Without loss of generality we assume that $c=(1\,2\dots n)$. Define
\begin{eqnarray*}
x_{1} & := & (1\,n)(2\,n-1)\dots(\frac{n-1}{2}\,\frac{n+3}{2}),\\
x_{2} & := & (2\,n)(3\,n-1)\dots(\frac{n+1}{2}\,\frac{n+3}{2}).
\end{eqnarray*}
 Then $x_{1}$ and $x_{2}$ are involutions in $A_{m}$ and $c=x_{1}\cdot x_{2}$.

Next consider a pair of cycles $c,c'\in\{c_{1},\dots,c_{k}\}$ of
even lengths $n$ and $n'$ respectively. Without loss of generality
we assume these have the form $c=(1\dots n)$ and $c'=(1'\dots n')$.
For $n>2$ define 
\[
y_{1}:=(1\,n-1)(2\,n-2)\dots(\frac{n-2}{2}\,\frac{n+2}{2}),
\]

\[
y_{2}:=(1\,n)(2\,n-1)\dots(\frac{n}{2}\,\frac{n+2}{2}),
\]
\[
z_{1}:=(1\,n)(2\,n-1)\dots(\frac{n}{2}\,\frac{n+2}{2}),
\]
\[
z_{2}:=(2\,n)(3\,n-1)\dots(\frac{n}{2}\,\frac{n+3}{2}).
\]
When $n=2$ define $y_{1}=z_{2}=1$ and $y_{2}=z_{1}=(1\,2)$. Furthermore,
define $y'_{1},\,y'_{2}$ etc to be the permutations as above, but
acting on the letters $1',\dots,n'$. Evidently all of the above elements
have order dividing two, but note that they are not necessarily even
permutations. In fact, when $n\equiv0$ mod $4$ , $y_{2},z_{1}\in A_{m}$
and $y_{1},z_{2}\in S_{m}\backslash A_{m}$. If $n\equiv2$ mod $4$
then the reverse situation holds. Naturally the same can be said for
$y'_{1}$ etc.

First assume that $n\equiv n'\equiv0$ mod $4$. Then $c\cdot c'=(y_{1}y_{2})\cdot(y'_{1}y'_{2})$.
But $y_{2}$ and $y'_{1}$ contain no common letters and hence commute.
We therefore rewrite $cc'=(y_{1}y'_{1})\cdot(y_{2}y'_{2})$, a product
of two involutions, $y_{1}y'_{1}$ and $y_{2}y'_{2}$ in $A_{m}$.
In an identical manner we write $cc'=(z_{1}z'_{1})\cdot(z_{2}z'_{2})$
when $n\equiv n'\equiv2$ mod $4$, and $cc'=(y_{1}z'_{1})\cdot($$y_{2}z'_{2})$
when $n\equiv0$ mod 4 and $n'\equiv2$ mod 4. Each bracketed expression
is an even permutation and hence in all cases we have written $cc'$
as the product of two involutions in $A_{m}$.

Next we consider cycles of length $3$ mod $4$. In particular, let
$c,c'\in\{c_{1},\dots,c_{k}\}$ be cycles of lengths $n$ and $n'$
respectively such that $n\equiv n'\equiv3$ mod $4$. Assuming that
$c=(1\dots n)$ and $c'=(1'\dots n')$, it follows that $c=x_{1}x_{2}$
and $c'=x'_{1}x'_{2}$, where $x_{1},x_{2}$ etc are as above. Note
however, that unlike the case above of $n\equiv1$ mod $4$, here
the $x_{i}$ are in $S_{n}\backslash A_{n}$. But $x_{2}$ and $x'_{1}$
commute and it follows that $cc'$ is indeed the product of the two
involutions $x_{1}x'_{1}$ and $x_{2}x'_{2}$ in $A_{m}$.

Lastly consider the cycle $c=(1\dots n)$ such that $n\equiv3$ mod
$4$, alone. If we assume additionally that $n\geq7$, then we can
easily write $c$ as the product of three involutions in $A_{m}$.
Namely 
\begin{equation}
c=\left(x_{1}(\frac{n-1}{2}\,\frac{n+3}{2})\right)\cdot\left((\frac{n-1}{2}\,\frac{n+3}{2})(2\,n)\right)\cdot\left((2\,n)x_{2}\right).\label{eq: 3 mod 4 as 3 involutions}
\end{equation}

Note that this analysis has ignored single cycles of length 3. This,
as we shall see below, is because such cycles must be considered within
the context of the whole element $g$.

Let us now return to our original element $g=c_{1}\dots c_{k}$. As
the $c_{i}$ commute, we assume that cycles of lengths congruent to
$3$ mod $4$ appear at the end of the decomposition, and are ordered
by increasing length. In particular, $|c_{k}|=3$ if and only if all
cycles of length 3 mod 4 are in fact of length 3 (and at least one
of these exists).

Assume for now that $n_{3}$, the number of cycles of length congruent
to $3$ mod $4$, is even. Then by the above, each cycle or pair of
cycles in $g$ can be decomposed as products of two involutions in
$A_{m}$. These decompositions have the form $c_{i}=t_{1_{i}}t_{2_{i}}$
or $c_{i}c_{i+1}=t_{1_{i}}t_{2_{i}}$, where $t_{i_{j}}\in A_{m}$
denote the involutions. These involutions then commute where necessary
to yield
\[
g=(t_{1_{1}}t_{2_{1}})\dots(t_{1_{p}}t_{2_{p}})=(t_{1_{1}}\dots t_{1_{p}})\cdot(t_{2_{1}}\dots t_{2_{p}}),
\]
 where $p=n_{1}+(n_{0}+n_{2})/2+n_{3}/2.$ Evidently $(t_{1_{1}}\dots t_{1_{p}})$
and $(t_{2_{1}}\dots t_{2_{p}})$ are involutions in $A_{m}$ and
hence $g$ is strongly real. 

Now consider the remaining case, where $n_{3}$ is odd. Firstly note
that $gc_{k}^{-1}$ is strongly real by the above method. In particular
there exist involutions $t_{i_{j}}\in A_{m}$ such that
\[
gc_{k}^{-1}=(t_{1_{1}}\dots t_{1_{q}})\cdot(t_{2_{1}}\dots t_{2_{q}}),
\]
 where $q=n_{1}+(n_{0}+n_{2})/2+(n_{3}-1)/2.$ If $|c_{k}|>3$, then
$c_{k}$ can be written as a product of three involutions as shown
by (\ref{eq: 3 mod 4 as 3 involutions}). Denote this decomposition
$c_{k}=s_{1}s_{2}s_{3}$, for involutions $s_{i}\in A_{m}$. It then
follows that $g=(t_{1_{1}}\dots t_{1_{q}}s_{1})\cdot(t_{2_{1}}\dots t_{2_{q}}s_{2})\cdot s_{3}$
and iw$(g)\leq3$. \\
This leaves the case where $|c_{k}|=3$. Write $c_{k}=(1\text{\,2\,3)=}(1\,2)(1\,3)$.
If $g=c_{k}$ then $c_{k}=\left((1\,2)(4\,5)\right)\cdot\left((4\,5)(1\,3)\right)$
and $g$ is strongly real. Suppose instead that $gc_{k}^{-1}\neq1$.
Then one of the elements $(t_{1_{1}}\dots t_{1_{q}})$ and $(t_{2_{1}}\dots t_{2_{q}})$
is non-trivial, say the latter. Let $(i\,j)$ be a transposition in
its cycle decomposition. This commutes with $c_{k}$ and hence we
can reorder the transpositions to give
\[
g=((t_{1_{1}}\dots t_{1_{q}})\cdot(t_{2_{1}}\dots t_{2_{q}}(i\,j)(1\,2))\cdot((i\,j)(1\,3)).
\]

This is a product of three involutions in $A_{m}$.
\end{proof}
\smallskip{}
It is worth noting the following corollary of the proof above.
\begin{cor}
Let $g\in A_{m}$ for $m\geq5.$ Suppose that $g$ has at least 2
fixed points or an even number of cycles with length $3\text{ mod }4$.
Then $g$ can be written as the product of 2 involutions in $A_{m}$.
\end{cor}
\begin{proof}
From the proof of Theorem \ref{An width at most 3} it is sufficient
to show that if $c$ is a cycle of length $n\equiv3\text{ mod }4$
such that $c\in A_{n+k}$ for $k\geq2$, then $c=t_{1}t_{2}$ for
involutions $t_{i}\in A_{n+k}.$ \\
Without loss of generality, take $c=(1\,2\dots n)$ with fixed points
$n+1$ and $n+2.$ We can write $c$ as a product of two involutions
as follows
\begin{eqnarray*}
(1\,2\,\dots n) & = & \left((1\,n)(2\,n-1)\dots(\frac{n-1}{2}\,\frac{n+3}{2})(n+1\,n+2)\right)\\
 &  & \cdot\left((n+1\,n+2)(2\,n)(3\,n-1)\dots(\frac{n+1}{2}\,\frac{n+3}{2})\right).
\end{eqnarray*}
\end{proof}

\subsection{Sporadic Groups}

In this section we prove the following result.
\begin{thm}
\label{thm:sporadic width at most 3}Let G be a sporadic finite simple
group. If $G\in\{J_{1},\,J_{2}\}$ then G has involution width 2,
otherwise G has width 3.
\end{thm}
It was found by Suleiman (\cite{Ibrahim Suleiman 2008}, see also
Theorem \ref{Thm: Strongly real G}) that only two sporadic groups,
namely $J_{1}$ and $J_{2}$ are strongly real. Furthermore, \cite{Ibrahim Suleiman 2008}
lists the non-strongly real classes of every other sporadic group.
We compute the involution width of these classes by calculating the
structure constants, which are defined as follows.
\begin{defn}
Let $G$ be a finite group and let $C_{i}$, $1\leq i\leq m$ be conjugacy
classes in $G$. For $g\in G$, define the \textit{structure constant}
$\eta(C_{1},\dots,C_{m},g^{G})$ to be the number of $m$-tuples $(g_{1},\dots,g_{m})\in C_{1}\times\dots\times C_{m}$
such that $g_{1}\dots g_{m}=g$.
\end{defn}
The structure constants can be computed using the following well known
formula.
\begin{thm}
\label{thm:Burnsides formula}(\cite{Arad and Herzog}, 1.10). Let
$G$ be a finite group and let $C_{i}$ , $1\leq i\leq m$ be conjugacy
classes in $G$. For $g\in G$, 
\[
\eta(C_{1},\dots,C_{m},g^{G})=\frac{|G|}{|C_{G}(g_{1})|\dots|C_{G}(g_{m})|}\sum_{\chi}\frac{\chi(g_{1})\dots\chi(g_{m})\chi(g^{-1})}{\chi(1)^{m-1}},
\]
 where the sum is over all the irreducible characters $\chi$ of $G$.
\end{thm}
Naturally $\eta(C_{1},\dots,C_{m},g^{G})$ is a non-negative integer.
However for our purposes we will only be concerned with showing that
$\eta(C_{1},\dots,C_{m},g^{G})\neq0$. Therefore to simplify proceedings,
we will predominantly compute 
\[
\kappa(C_{1},\dots,C_{m},g^{G})=\sum_{\chi}\frac{\chi(g_{1})\dots\chi(g_{m})\chi(g^{-1})}{\chi(1)^{m-1}},
\]
 which we call the \textit{normalised structure constant}.\\
\\
\textbf{Proof of Theorem }\ref{thm:sporadic width at most 3}:

\begin{proof}[\unskip\nopunct]
The full character tables are known for all sporadic groups (\cite{the atlas}) and hence the structure constants can be calculated explicitly. We use the software package GAP \cite{GAP} and a short function that checks whether $\eta(C_{1},C_{1},C_{1},C_{2})$ is nonzero for $C_{1}$, a class of involutions and $C_{2}$, a chosen non-strongly real class. For almost all the sporadic groups, we find that every such conjugacy class $C_{2}$ is contained in the cube of the $2A$ class. For the exceptions $He$, $Co_{2}$, $Fi_{22}$, $Fi_{23}$ and $BM$, we find instead that each class  $C_{2} \subset (2B)^{3}$. 
\end{proof}

\section{Non-Unitary Classical Groups\label{sec:Simple-Groups-of pre lime}}

\subsection{Preliminary Material\label{subsec:Preliminary-Material for general Lie Type Groups}}

 The remainder of this paper will be devoted to proving Theorem~\ref{mainthm}
for the simple groups of Lie type. Throughout this section, $\boldsymbol{G}$
denotes a simple linear algebraic group over an algebraically closed
field of characteristic $p>0$. Let $F\,:\,\boldsymbol{G}\rightarrow\boldsymbol{G}$
be a Frobenius endomorphism. Then the set of fixed points $G:=\boldsymbol{G}^{F}$
is the associated finite group of Lie type. Usually we take $\boldsymbol{G}$
to be simply connected, in which case, with a small number of exceptions,
$G$ is quasisimple and $\overline{G}=G/Z(G)$ is the simple group
of interest.

The general procedure for proving Theorem~\ref{mainthm} for a simple
group of Lie type $\overline{G}$ is as follows. We aim to pick regular
semisimple elements $x$ and $y$ in particular classes of $G$, such
that 
\begin{enumerate}
\item The projections $\overline{x},\overline{y}$ are strongly real in
the simple group $\overline{G}$.
\item $G\backslash Z(G)\subset x^{G}\cdot y^{G}.$
\end{enumerate}
From (2) it will follow that $\overline{G}\backslash\{1\}\subset\overline{x}^{\overline{G}}\cdot\overline{y}^{\overline{G}}$
and hence iw$(\overline{G})\leq\text{iw}(\overline{x})+\text{iw}(\overline{y})\leq4$
by (1).

The proof that $G\backslash Z(G)\subseteq x^{G}\cdot y^{G}$ involves
calculating structure constants using Theorem \ref{thm:Burnsides formula}.
However unlike the sporadic case, complete character tables are currently
unavailable for most groups of Lie type. We therefore rely substantially
on estimates of character values. In particular, technical details
of the characters of unitary groups form a substantial part of this
work (see Sec. \ref{sec:Unitary-Groups}). 

Before proceeding, we briefly recall some results on ordinary representations
of finite groups of Lie type (see Carter \cite{Carter rep theory}
and Digne and Michel \cite{Digne and Michel textbook}).

Let $\boldsymbol{T}$ be an an $F$-stable maximal torus in $\boldsymbol{G}$
and $\theta\in$Irr$(\boldsymbol{T}^{F})$. One can define a corresponding
Deligne-Lusztig character $R_{T}^{G}(\theta)$ of the fixed point
group $G=\boldsymbol{G}^{F}$ such that $R_{T}^{G}(\theta)\in\mathbb{Z}[\text{Irr}(G)]$
(\cite{Carter rep theory}, 7.2). Moreover, for $\chi\in\text{Irr}(G)$,
there is a pair $(\boldsymbol{T},\theta$) such that $\chi$ occurs
in the decomposition of $R_{T}^{G}(\theta)$.

We shall work in the setting of Lusztig series of characters. For
this let $\boldsymbol{G}^{*}$ be the reductive group in duality with
$\boldsymbol{G}$, with corresponding Frobenius endomorphism $F^{*}:\,\boldsymbol{G}^{*}\rightarrow\boldsymbol{G}^{*}$.
The dual finite group of Lie type is then the fixed point set $G^{*}:=\boldsymbol{G}^{*F^{*}}$
(see (\cite{Digne and Michel textbook}, 13.10)). There is a bijective
correspondence between $G$-conjugacy classes of pairs $(T,\theta)$
as defined above, and $G^{*}$-classes of pairs $(T^{*},s)$, where
$T^{*}\leq G^{*}$ is a maximal torus dual to $T$ and $s\in T^{*}$
is semisimple (\cite{Digne and Michel textbook}, 13.12). If $(T^{*},s)$
corresponds to $(T,\theta)$ in this manner, we re-label $R_{T}^{G}(\theta)$
as $R_{T^{*}}^{G^{*}}(s)$. The Lusztig series $\mathcal{E}(G,(s))$
is then defined to be the set of irreducible constituents of the $R_{T^{*}}^{G^{*}}(s)$
for $T^{*}$ running over maximal tori of $G^{*}$ containing $s$.
We note the following lemma. 
\begin{lem}
(\cite{Guralnick and Malle}, 3.2).\label{lem:G and M lemma on the existence of a character in a non trivial Lusztig series}
Let $g\in G$ be semisimple and $\chi\in$Irr$(G)$ with $\chi(g)\neq0$.
Then there exists a maximal torus $T\ni g$ of $G$ and $s\in T^{*}\leq G^{*}$
such that $\chi\in\mathcal{E}(G^{F},(s))$.
\end{lem}
Of the characters of finite groups of Lie type, the best understood
are the so-called unipotent characters. These are the elements of
the Lusztig series $\mathcal{E}(G,(1))$. 
\begin{prop}
(Jordan decomposition of characters, (\cite{Digne and Michel textbook},
13.23)).\label{prop:(Jordan-decomposition-of characters and degree formula}
For any semisimple element $s\in G^{*}$, there is a bijection $\psi_{s}$
from $\mathcal{E}(G,(s))$ to $\mathcal{E}(C_{\boldsymbol{G}^{*}}(s)^{F^{*}},(1))$.
The degree of any character $\chi\in\mathcal{E}(G,(s))$ is given
by the formula
\[
\chi(1)=|G^{*}:C_{\boldsymbol{G}}(s)^{F^{*}}|_{p'}\cdot(\psi_{s}(\chi))(1).
\]
\end{prop}
Note that some caution is required when the centre of $\boldsymbol{G}$
is disconnected. If the centre of $\boldsymbol{G}$ is connected,
then the group $C_{\boldsymbol{G}^{*}}(s)$ is connected reductive
for any semisimple $s\in\boldsymbol{G}^{*}$. If however the centre
of $\boldsymbol{G}$ is disconnected then the definition of the unipotent
characters $\mathcal{E}(C_{\boldsymbol{G}^{*}}(s)^{F^{*}},1)$ needs
to be generalised to disconnected groups (\cite{Digne and Michel textbook},
13.23).

Recall that for a prime $p$ and a character $\chi\in$Irr$(G)$,
we say that $\chi$ has $p$-defect zero if $p$ does not divide $\frac{|G|}{\chi(1)}$.
The study of $p$-defect zero characters plays an important role in
this work and we will often use the following result of Brauer.
\begin{thm}
\label{thm:Brauers theorem on prime defect}((\cite{Isaacs Character Theory of finite groups},
8.17). Let $p$ be a prime and $\chi\in\text{Irr}(G)$ a character
of $p$-defect zero. Then $\chi(g)=0$ whenever $p\mid o(g)$.
\end{thm}
We will predominantly apply the above theorem with respect to particular
prime factors of $|G|$ known as primitive prime divisors.
\begin{defn}
\label{def:ppd}Let $q=p^{f}$ for some prime $p$ and let $n\geq2$.
Then a primitive prime divisor, denoted ppd$(q,n)$, is a prime $r$
such that $r|(q^{n}-1)$ but $r\nmid(q^{k}-1)$ for $1\leq k\leq n-1$.
\\
A primitive prime divisor always exists for $n>2$ by Zsigmondy's
Theorem (\cite{Birkhoff for Zsgimondy}, Thm. V) except when $(n,q)=(6,2)$.
\end{defn}
A final result that will be useful in proving step (2) of the procedure
detailed above is the following due to Gow.

\begin{thm}
\label{thm:(Gow-)Theorem for ss elements}(\cite{Gow ss elements results},
2). Let $G$ be a finite simple group of Lie type and let $L_{1}$
and $L_{2}$ be conjugacy classes of $G$ consisting of regular semisimple
elements. Then any non-identity semisimple element of $G$ is expressible
as a product $xy$, where $x\in L_{1}$ and $y\in L_{2}$.
\end{thm}

First we consider the non-unitary classical groups, that is the linear,
orthogonal and symplectic finite simple groups. These are denoted
$PSL_{n}(q)$, $P\Omega_{n}^{\pm}(q)$ for $n$ even, $P\Omega_{n}(q)$
for $n$ odd and $PSp_{n}(q)$. We see that for some families, the
involution width problem has already been considered and we quote
results where appropriate. 

\subsection{Linear Groups}

In \cite{Gustavson et al} Gustavson et al. proved that the group
of matrices in $GL_{n}(K)$ ($K$ a field) of determinant $\pm1$
has involution at most 4. The problem of the special linear group
was then addressed by Kn$\ddot{\text{u}}$ppel  and Nielsen \cite{Knuppel and Knielsen}. 
\begin{thm}
\label{thm:sl 4 ref}(\cite{Knuppel and Knielsen}, Thm. A). Let $G=SL_{n}(q)$,
$n\geq3$. Then $G$ has involution width $3$ if and only if at least
one of the following holds
\end{thm}
\begin{enumerate}
\item \textit{$n=4$,}
\item \textit{$q\in\{2,3,5\}$ and $n\not\equiv2$ mod $4$,}
\item \textit{$n=3$, and either char$(\mathbb{F}_{q})=3$ or $x^{2}+x+1$
is irreducible over $\mathbb{F}_{q}$ .}
\end{enumerate}
\textit{Otherwise $G$ has involution width 4.}
\begin{cor}
\label{cor:iw of PSL is at most 4}Let $\overline{G}=PSL_{n}(q)$,
$n\geq2$ (not $PSL_{2}(2)$ or $PSL_{2}(3)$).
\end{cor}
\begin{enumerate}
\item \textit{Then \textup{iw}$(\overline{G})=2$ if and only if $n=2$
and $q\not\equiv3$ mod $4$. }
\item \textit{Suppose $(n,q)$ satisfies at least one of the conditions
1-3 of Theorem \ref{thm:sl 4 ref} or $(n,q)=(2,\,3\text{ mod}\,4)$.
Then \textup{iw}$(\overline{G})=3$.}
\item \textit{In any case, \textup{iw}$(\overline{G})\leq4$.}
\end{enumerate}
\begin{proof}
Part 1 follows from Theorem \ref{Thm: Strongly real G}. Part 2 follows
from Theorem \ref{thm:sl 4 ref} together with (\cite{Arad and Herzog},
Thm 4.2), which shows that $C^{3}=PSL_{2}(q)$ for any non-identity
class $C$. Part 3 is immediate from Theorem \ref{thm:sl 4 ref} . 
\end{proof}

\subsection{Orthogonal Groups}

The orthogonal groups are another classical family that have been
considered in the literature. Building on a number of results we prove
\begin{thm}
Let $n\ge7$ and let $\overline{G}=P\Omega_{n}^{\pm}(q)$ ($n$ even)
or $P\Omega_{n}(q)$ ($n$ odd).\label{thm: orthog width}
\end{thm}
\begin{enumerate}
\item \textit{If $q$ is odd then} iw$(\overline{G})\leq3$.
\item \textit{If $q$ is even then} iw$(\overline{G})\leq4.$
\end{enumerate}
\begin{proof}
Firstly assume that $q$ is odd. By (\cite{Knuppel and Thomsen 98},
8.5), iw$(\Omega_{n}(q))\leq3$. The result then holds in the simple
group after taking the natural map. Next assume that $q$ is even
(and hence $G=P\Omega_{2m}^{\epsilon}(q))$. By (\cite{Ellers and Gordeev},
Thm. 3) we know that any $g\in G$ can be written as a product of
two unipotent elements of $G$. But then, by (\cite{Johanna Ramo},
Thm. 1.2), such unipotent elements are strongly real and hence it
follows that iw$(g)\leq4$.
\end{proof}

\subsection{Symplectic Groups}

Recall from Theorem \ref{Thm: Strongly real G} that the symplectic
groups $PSp_{2n}(q)$ are strongly real if $q\not\equiv3$ mod $4$.
The next result considers $PSp_{2n}(q)$ for all values of $q$. Note
also that $PSp_{2}(q)\cong PSL_{2}(q)$ and thus by Corollary \ref{cor:iw of PSL is at most 4}
we can assume that $n\geq2$.
\begin{thm}
\label{thm:Sp 4 reflectional}Let $G=Sp_{2n}(q)$ with $n\geq2$.
Then $\overline{G}=PSp_{2n}(q)$ has involution width at most 4. 
\end{thm}
\begin{proof}
The work of Guralnick and Tiep (\cite{Guralnick and Tiep}, Sec. 2.2.3)
shows that as long as $(n,\,q)\notin\{(2,2),\,(2,3),\,(3,2)\}$ we
can find regular semisimple elements $x,\,y\in G$ such that $x^{G}\cdot y^{G}\supseteq G\setminus Z(G).$
Such elements are contained in tori $T_{1}$ of order $q^{n}+1$ and
$T_{2}$ of order $q^{n}-1$ respectively. The cases where $(n,\,q)\in\{(2,2),\,(2,3),\,(3,2)\}$
are easily dealt with using the character tables in \cite{the atlas},
so we assume $(n,\,q)$ is not one of these from now on.

Let $\overline{g}$ be the image of $g\in G$ under the canonical
map $G\rightarrow\overline{G}$. By (\cite{Malle and Saxl and Weigel}
, 2.3(c)), $\overline{x}$ is strongly real so iw$(\overline{x})=2.$
We claim that $\overline{y}$ is also strongly real, from which the
theorem will follow.\\
Now $T_{2}$ is a cyclic torus, conjugate to a Singer cycle $GL_{1}(q^{n})<GL_{n}(q)<G$.
Let $z\in T_{2}$ be a generator of the torus. We shall show that
a conjugate of $z$ is strongly real and thus so is $z$ itself.

Let $e_{1},\dots e_{n},f_{1},\dots,f_{n}$ be a standard basis for
the symplectic space, where $(e_{i},f_{j})=\delta_{ij}$ and $V_{e}=\langle e_{1},\dots,e_{n}\rangle$,
$V_{f}=\langle f_{1},\dots,f_{n}\rangle$ are totally isotropic subspaces.
Let $H\cong GL_{n}(q)$ be the stabiliser in $G$ of $V_{e},V_{f}$.
Explicitly,
\[
H=\{\,\begin{pmatrix}A & 0\\
0 & A^{-T}
\end{pmatrix}|\,A\in GL_{n}(q)\}
\]
 and we let $z\in H$ have matrix representation $\begin{pmatrix}B & 0\\
0 & B^{-T}
\end{pmatrix}$ for some $B\in GL_{n}(q)$.

Now let $t\in Sp_{2n}(q)$ be the map that sends $e_{i}\rightarrow f_{i}$
and $f_{i}\rightarrow-e_{i}$, $1\leq i\leq n$. Evidently this map
normalises the subgroup $H$ and swaps $V_{e}$ and $V_{f}$. Explicitly,
$t$ is given by the matrix $\begin{pmatrix}0 & -I_{n}\\
I_{n} & 0
\end{pmatrix}$ and thus conjugation of $z\in H$ is given by 
\[
t^{-1}zt=\begin{pmatrix}0 & I_{n}\\
-I_{n} & 0
\end{pmatrix}\begin{pmatrix}B & 0\\
0 & B^{-T}
\end{pmatrix}\begin{pmatrix}0 & -I_{n}\\
I_{n} & 0
\end{pmatrix}=\begin{pmatrix}B^{-T} & 0\\
0 & B
\end{pmatrix}.
\]

Clearly, $\overline{t}$ is an involution and hence to show $\overline{z}$
is strongly real it suffices to show that $B$ is conjugate in $GL_{n}(q)$
to a symmetric matrix.

As $B$ is a Singer cycle in $GL_{n}(q)$ it has a minimal polynomial
$f$ of degree $n$. Furthermore $B$ is cyclic and conjugate to the
companion matrix $C(f)$. In (\cite{Brawley Teitloff}, Thm. 7), Brawley
and Teitloff prove that if $f\in\mathbb{F}_{q}[x]$ is monic of degree
$n$ and $f$ is not square in $\mathbb{F}_{q}[x]$ then $C(f)$ is
indeed similar to a symmetric matrix. 

In conclusion, $z$ is conjugate to some $h\in H$ which is inverted
by $t$. Thus $\overline{z}$, and hence also $\overline{y}$, is
strongly real as claimed. 
\end{proof}

\section{\label{sec:Unitary-Groups}Unitary Groups}

In this section we complete the proof of Theorem~\ref{mainthm} for
classical groups by proving
\begin{thm}
\label{thm:inv wwidth of PSU}Let $\overline{G}=PSU_{n}(q)$ with
$n\geq3$ and $(n,q)\neq(3,2).$ Then $\overline{G}$ has involution
width at most 4. 
\end{thm}
After some preliminary material, we treat the cases of even and odd-dimensional
unitary groups separately. Substantially more effort is required for
the odd-dimensional case.

\subsection{Preliminaries\label{subsec:unitary groups prelims}}

Regard $G=SU_{n}(q)=\boldsymbol{G}^{F}$ where $\boldsymbol{G}=SL_{n}(\overline{\mathbb{F}}_{q})$
and $F$ is a Frobenius endomorphism as in Section \ref{subsec:Preliminary-Material for general Lie Type Groups}.
Explicitly, $F$ is the standard Frobenius map that raises matrix
entries to the $q^{\text{th }}$power, composed with the inverse-transpose
map. Let $\Phi$ denote the root system of $\boldsymbol{G}$ of type
$A_{n-1}$. Then the maximal tori of $G$ and its dual $G^{*}=PGU_{n}(q)$
correspond to $F$-conjugacy classes of the Weyl group $W=W(\Phi)\cong S_{n}.$
These in turn correspond to $S_{n}$-orbits in the non-trivial coset
of Aut$(\Phi)=S_{n}\times\{\pm1\}$ (\cite{Tiep Shalev warring},
2.2). Such an orbit is given by a conjugacy class in $S_{n}$ and
therefore a partition $(a_{1},\dots a_{k})$, of $n$. We denote the
corresponding maximal torus by $T_{a_{1},\dots,a_{k}}$. Recall that
there exists a torus $T^{*}<G^{*}$ that is dual to $T$ in the sense
that the Frobenius actions on their character groups are mutually
transpose (\cite{Tiep Shalev warring}, 2.2). This duality gives a
bijective correspondence between types of maximal tori in $G$ and
$G^{*}$. In particular, $T_{a_{1},\dots,a_{k}}$ and its dual denoted
$T_{a_{1},\dots,a_{k}}^{*}$ have the same order
\begin{equation}
|T_{a_{1},\dots,a_{k}}|=|T_{a_{1},\dots,a_{k}}^{*}|=\frac{(q^{a_{1}}-(-1)^{a_{1}})\dots(q^{a_{k}}-(-1)^{a_{k}})}{(q+1)}.\label{eq:Unitary tori order}
\end{equation}
Note that these orders, as well as those for all maximal tori of Lie
type groups, can be calculated using the methods in \cite{Carter Conjugacy classes in the weyl group}.

Firstly we define what it means for two tori to be weakly orthogonal
and explore how this affects the values of characters evaluated on
such tori. This concept was first seen in \cite{Malle and Saxl and Weigel}
and then given a formal definition in \cite{Tiep Shalev warring}.
\begin{defn}
We call two maximal tori $T_{1}=\boldsymbol{T}_{1}^{F}$ and $T_{2}=\boldsymbol{T}_{2}^{F}$
weakly orthogonal if 
\[
\boldsymbol{S}_{1}^{*F^{*}}\cap\boldsymbol{S}_{2}^{*F^{*}}=1
\]
 for every pair $\boldsymbol{S}_{1}^{*},\boldsymbol{S}_{2}^{*}$ of
$F^{*}$-stable conjugates of $T_{1},T_{2}$.
\end{defn}
If $x$ and $y$ are regular semisimple elements of $G=\boldsymbol{G}^{F}$
with centralizers $T_{1}$ and $T_{2}$ respectively, we say $x$
and $y$ are weakly orthogonal if and only if $T_{1}$ and $T_{2}$
are weakly orthogonal. 

For unitary groups there are known collections of weakly orthogonal
tori pairs:
\begin{prop}
\label{prop:().-wo tori in su}(\cite{Tiep Shalev warring}, 2.3.2).
Let $G=SU_{n}(q)$. For $0\leq a\leq n-1$, the maximal tori $T_{n}$
and $T_{1,a,n-1-a}$ are weakly orthogonal. If $2\leq a\leq n-2$,
the maximal tori $T_{1,n-1}$ and $T_{a,n-a}$ are weakly orthogonal.
\end{prop}
Here is a result concerning the vanishing of characters on weakly
orthogonal elements.
\begin{prop}
(\cite{Tiep Shalev warring}, 2.2.2).\label{prop:weakly orthogonal implies only unipotent.}
If $x$ and $y$ are weakly orthogonal regular semisimple elements
of $G$, and $\chi$ is an irreducible character of $G$ such that
$\chi(x)\chi(y)\neq0$, then $\chi$ is unipotent. 
\end{prop}
It is thus useful when computing structure constants to consider weakly
orthogonal elements as this allows for a reduction down to the unipotent
characters, of which the most is known. Recall from Section \ref{subsec:Preliminary-Material for general Lie Type Groups}
that the unipotent characters are exactly the elements of the Lusztig
series $\mathcal{E}(G,(1))$. 

For classical groups we have a combinatorial description for such
characters (\cite{Carter rep theory}, 13.8). In particular, for a
root system of type $A_{n-1}$, unipotent characters correspond to
partitions of $n$. 

Let $\lambda=(\lambda_{1}\geq\dots\geq\lambda_{k})$ be a partition
of $n$ and let $[\lambda]$ be the associated Young diagram. Denote
the set of hooks of $[\lambda]$ by $H$, and the length of a hook
$h\in H$ by $l(h)$. Also, define $a(\lambda):=\sum_{i<j}\text{min}(\lambda_{i},\lambda_{j})$
and the integral polynomial 
\begin{equation}
\rho_{\lambda}(x)=\frac{(x^{n}-1)(x^{n-1}-1)\dots(x-1)}{\prod_{h\in H}(x^{l(h)}-1)}\cdot x^{a(\lambda)}.\label{eq:degree formula of unipotent characters}
\end{equation}
 The degree of the corresponding unipotent character denoted $\chi_{\lambda}\in$Irr$(SU_{n}(q))$,
is then given by $\chi_{\lambda}(1)=\pm\rho_{\lambda}(-q)$, where
the sign is chosen such that the highest coefficient is positive.
For example if $\lambda=(n-1,1)$ where $n$ is odd, then $\chi_{\lambda}(1)=q(q^{n-1}-1)/(q+1)$.
Note also that the unipotent character of $SL_{n}(q)$ associated
to $\lambda$ has degree $\rho_{\lambda}(q)$.

Recall from Proposition \ref{prop:(Jordan-decomposition-of characters and degree formula}
that a general character $\chi\in\text{Irr}(G)$ corresponds to a
member of a Lusztig series $\mathcal{E}(C_{\boldsymbol{G}^{*}}(s)^{F^{*}},1)$,
where $s\in G^{*}$ is semisimple. Centralizers of semisimple elements
are well understood in unitary groups (\cite{Carter Centralizers paper},
Sec. 8) and we shall to only consider the case where $C_{\boldsymbol{G}^{*}}(s)$
is connected. In particular, connected centralizers in $GU_{n}(q)$
of semisimple elements are direct products of factors of type $GL_{n_{i}}(q^{a_{i}})$
and $GU_{n_{i}}(q^{a_{i}})$. We then consider the image of these
products in $G^{*}=PGU(q)$. A unipotent character of $C_{\boldsymbol{G}^{*}}(s)^{F^{*}}$
restricted to one of these components then corresponds to a partition
$\lambda^{i}\vdash n_{i}$ as explained above. Hence letting these
partitions form the entries of a multi-partition denoted $\overline{\lambda}$,
it follows that a given $\chi\in\text{Irr}(G)$ is completely determined
by $s$ and $\overline{\lambda}$. We therefore adopt the notation
$\chi=\chi_{s,\overline{\lambda}}$ to refer to this character. We
however follow the usual notation for unipotent characters that we
have already used above. That is, when $s=1$ we write $\chi_{1,\overline{\lambda}}=\chi_{\lambda}$.

Note that the hooks contained in a diagram $[\lambda]$ match those
of the diagram of the conjugate partition $[\lambda']$. Consequently,
the degrees $\chi_{\lambda}(1)$ and $\chi_{\lambda'}(1)$ differ
only by a scalar factor $q^{a(\lambda')-a(\lambda)}$. 

We can further relate pairs of characters in Irr$(G)$ using the notion
of Alvis-Curtis duality. For a reductive group $\boldsymbol{G}$,
the Alvis-Curtis duality functor $D_{\boldsymbol{G}}$ sends any $\chi\in$Irr$(\boldsymbol{G}^{F})$
to another irreducible character of $\boldsymbol{G}^{F}$ up to sign
(\cite{Digne and Michel textbook}, 8.8). This duality is particularly
important as we will often only have to find character values for
one character in each pairing.
\begin{lem}
\label{lem dual charracters fixed on reg sem si-1}Let $g\in G$ be
regular semisimple. Then $|\chi_{s,\overline{\lambda}}(g)|=|D_{\boldsymbol{G}}(\chi_{s,\overline{\lambda}})(g)|$. 
\end{lem}
\begin{proof}
This follows from part of the proof of (\cite{Guralnick and Tiep},
2.3) which we will summarise here. By the definition of the functor
$D_{\boldsymbol{G}}$ (\cite{Digne and Michel textbook}. 8.8) it
is clear that $D_{\boldsymbol{T}}(\lambda)=\lambda$ for any $F$-stable
torus $\boldsymbol{T}$ and $\lambda\in$Irr$(\boldsymbol{T}^{F})$.
Recall that as $g$ is regular semisimple, $C_{\boldsymbol{G}}(g)^{F}=T_{1}$
for some maximal torus $T_{1}<G$. Applying this and (\cite{Digne and Michel textbook},
Cor 8.16), we see that 
\[
|D_{\boldsymbol{G}}(\chi_{s,\overline{\lambda}})(g)|=|(D_{C_{\boldsymbol{G}}(g)}\circ\text{Res}_{C_{\boldsymbol{G}}(g)^{F}}^{G})(\chi_{s,\overline{\lambda}})(g)|=|\chi_{s,\overline{\lambda}}(g)|.
\]
\end{proof}
Recent work of Guralnick, Larsen and Tiep \cite{GLT Preprint} descibes
the Alvis-Curtis duality of unitary groups.
\begin{lem}
(\cite{GLT Preprint}, 5.3).\label{lem tiep and bob result} Let $\tilde{G}=GU_{n}(q)$
and $\chi_{s,\overline{\lambda}}\in\text{Irr}(\tilde{G})$, where
$s\in\tilde{G}^{*}$ is semismple and $\overline{\lambda}=(\lambda^{1},\lambda^{2},\dots)$
is a multi-partition, as described above. Let $\overline{\mu}=((\lambda')^{1},(\lambda')^{2},\dots)$.
Then $D_{\boldsymbol{G}}(\chi_{s,\overline{\lambda}})=\chi_{s,\overline{\mu}}$.
\end{lem}
As we are working in $G=\boldsymbol{G}^{F}=SU_{n}(q)$, we consider
the extension of characters to $\tilde{G}=\tilde{\boldsymbol{G}}^{F}=GU_{n}(q)$
before applying Lemmas \ref{lem dual charracters fixed on reg sem si-1}
and \ref{lem tiep and bob result}. Recall that $G^{*}=PGU_{n}(q)$
and $\tilde{G}^{*}=GU_{n}(q)$, and let $\pi:\tilde{G}^{*}\rightarrow G^{*}$
denote the usual projection map. For a semisimple element $s\in G^{*}$,
there exists $\tilde{s}\in\tilde{G}^{*}$ such that $\pi(\tilde{s})=s$
and we have the Lusztig series $\mathcal{E}(G,(s))\subseteq\text{Irr}(G)$
and $\mathcal{E}(\tilde{G},(\tilde{s}))\subseteq\text{Irr}(\tilde{G})$
. Furthermore,
\[
\mathcal{E}(G,(s))=\{\chi\in\text{Irr}(G)\,|\,\chi\,\text{occurs in}\,\tilde{\chi}|_{G}\,\text{for some}\,\tilde{\chi}\in\mathcal{E}(\tilde{G},(\tilde{s}))\}.
\]

To check if $\text{\ensuremath{\chi\in\text{Irr}}}(G)$ extends to
a a character of $\tilde{G}$ we use the following result.
\begin{lem}
(\cite{Isaacs Character Theory of finite groups}, 11.22). \label{lem:cyclic extensions and extendible characters}Let
$H$ be a finite group and let $N\trianglelefteq G$ such that $G/N$
is cyclic. Let $\chi\in\text{Irr}(N)$ be $G$-invariant. Then $\chi$
is extendible to $G$.
\end{lem}
To obtain information about the splitting of a character $\tilde{\chi}\in\mathcal{E}(\tilde{G},(\tilde{s}))$
upon restriction from $\tilde{G}$ to $G$, we study the centralisers
$\tilde{H}=C_{\tilde{\boldsymbol{G}}^{*}}(\tilde{s})$ and $H=C_{\boldsymbol{G}^{*}}(s)$.
Recall from Proposition \ref{prop:(Jordan-decomposition-of characters and degree formula}
that $\psi_{s}(\tilde{\chi})$ is a unipotent character of $\tilde{H}^{F}$.
Let $\Gamma$ denote the stabiliser of $\psi_{s}(\tilde{\chi})$ under
the action of $H^{F}/H^{\circ F}$. Then $\tilde{\chi}|_{G}$ splits
into $|\Gamma|$ constituents. We will only consider Lusztig series
$\mathcal{E}(G,(s))$ where $H=C_{\boldsymbol{G}^{*}}(s)$ is connected
and so the characters restrict irreducibly. 

\subsection{$G=SU_{n}(q)$, n even\label{subsec:n-even}}

In this section we complete the proof of Theorem \ref{thm:inv wwidth of PSU}
for even dimensions $n\geq4$. 
\begin{prop}
\label{prop:picking x and y}Let $G=SU_{n}(q)$, $n\geq4$ even. There
exist regular semisimple elements $x\in T_{n}$ and $y\in T_{1,1,n-2}$
such that $\overline{x}$ and $\overline{y}$ are strongly real in
the simple group $\overline{G}=PSU_{n}(q)$.
\end{prop}
\begin{proof}
The maximal torus $T_{n}<G$ has order $\frac{q^{n}-1}{q+1}$ and
is cyclic (\cite{Gager Thesis}, 4.5). Choose $x\in T_{n}$ of prime
order $r$, where $r=$ ppd$(q,\,n)$ if $n\equiv0$ mod $4$ and
$r=$ ppd $(n/2,\,q)$ if $n\equiv2$ mod $4$. Note in both cases
$r$ divides $q^{\frac{n}{2}}+(-1)^{\frac{n}{2}}.$ Such primitive
prime divisors exist (see Definition \ref{def:ppd}) Evidently $x$
is regular as the order formula (\ref{eq:Unitary tori order}) prevents
it from lying in any other maximal torus. 

To show that $x$ is strongly real in $\overline{G}$ we look for
containment in a strongly real subgroup. First note that $G$ has
an $SL_{2}(q^{\frac{n}{2}})$ subgroup containing $x$. Indeed, when
$n\equiv2$ mod $4$, $SU_{2}(q^{\frac{n}{2}})\cong SL_{2}(q^{\frac{n}{2}})$
is a field extension subgroup of $SU_{n}(q),$ whereas if $n\equiv0$
mod $4$ we have the embedding $SL_{2}(q^{\frac{n}{2}})<SL_{\frac{n}{2}}(q^{2})<SU_{n}(q)$.

In even characteristic, $PSL_{2}(q^{\frac{n}{2}})=SL_{2}(q^{\frac{n}{2}})$
and this subgroup is strongly real by Theorem \ref{Thm: Strongly real G}.
In odd characteristic we can find $g\in SL_{2}(q^{\frac{n}{2}})$
such that $g^{2}=-1$ and $x^{g}=-x^{-1}$. It follows that $\overline{g}$
is an involution in $\overline{G}$ and hence $\overline{x}$ is strongly
real. 

A similar method is used for $y$ where now we want to pick a regular
element in a torus of size $|T_{1,1,n-2}|=(q^{n-2}-1)(q+1)$. We take
$y$ as the product of a regular element $y_{1}$ in the cyclic torus
$T_{n-2}<SU_{n-2}(q)$ and $y_{2}=$diag$(\mu,\mu^{-1})\in T_{1,1}\leq SU_{2}(q)$
where $\mu\neq\mu^{-1}$ and $\mu\overline{\mu}=1$. Here matrices
in $SU_{2}(q)$ are written with respect to an orthonormal basis. 

A regular element $y_{1}$ exists in $SU_{n-2}(q)$ exactly as we
picked $x$ above, and evidently $C_{SU_{2}(q)}(y_{2})=T_{1,1}$.
It follows that $C_{SU_{n}(q)}(y)=T_{1^{2},n-2}$ and so $y$ is regular.
As above, there exists an element $g_{1}\in SU_{n-2}(q)$ such that
$g_{1}^{2}=-1$ and $g_{1}y_{1}g_{1}^{-1}=-y_{1}^{-1}$. Let $g_{2}=\begin{pmatrix}0 & -1\\
1 & 0
\end{pmatrix}\in SU_{2}(q)$. Then it follows that $g:=$diag$(g_{1},g_{2})$ is an involution
in $\overline{G}$, conjugating $\overline{y}$ to its inverse. Hence
$\overline{y}$ is strongly real.
\end{proof}
For the remainder of this section, we fix the elements $x,y\in G=SU_{n}(q)$
($n\geq4$ and even) as in Proposition \ref{prop:picking x and y}.
We now calculate the normalised structure constants using the formula
given in Theorem \ref{thm:Burnsides formula}. Recall, we wish to
show that for arbitrary $g\in G\backslash Z(G)$,
\[
\sum_{\chi\in\text{Irr}(G)}\frac{\chi(x)\chi(y)\chi(g^{-1})}{\chi(1)}\neq0.
\]
In the following, let $\Phi_{k}(z)$ denote the usual $k^{\text{th}}$
cyclotomic polynomial in the variable $z$.
\begin{lem}
\label{lem:reduction to 4 character lemma}Let $G=SU_{n}(q)$ where
$n\geq4$ is even. Suppose $\chi$ is an irreducible character of
$G$ such that $\chi(x)\chi(y)\neq0$. Then $\chi$ is unipotent and
corresponds to a partition $\lambda\in\{(n),(1^{n}),(n-1,1),(2,1^{n-2})\}$.
\end{lem}
\begin{proof}
By Proposition \ref{prop:().-wo tori in su}, $x$ and $y$ are weakly
orthogonal. It therefore follows from Proposition \ref{prop:weakly orthogonal implies only unipotent.}
that $\chi$ is unipotent. We first consider unipotent characters
such that $\chi_{\lambda}(x)\neq0$. Recall that $\chi_{\lambda}(1)=|\rho_{\lambda}(-q)|$
where the polynomial $\rho_{\lambda}(x)$ is given by the formula
(\ref{eq:degree formula of unipotent characters}), and $x$ has prime
order $r$, defined in the proof of Proposition \ref{prop:picking x and y}.
Clearly $r$ divides $\rho_{\lambda}(-q)$ if and only if $\Phi_{n}(-q)$
divides $\rho_{\lambda}(-q)$. It follows that either $\rho_{\lambda}(-q)_{r}=1$
or $\rho_{\lambda}(-q)_{r}=|G|_{r}$. But $\Phi_{n}(-q)$ divides
$\rho_{\lambda}(-q)$ if and only if $[\lambda]$ has no hooks of
length $n$. Hence $\chi_{\lambda}(x)\neq0$ implies that $[\lambda]$
has an $n$-hook and so the partition must be of the form $\lambda=(k,1^{n-k})$
for some $k$. 

The case for characters not vanishing on $y$ is similar as the order
of $y$ is divisible by a primitive prime divisor $s,$ of $\Phi_{n-2}(-q)$.
If $n\geq6$ then as above, either $\rho(-q)_{s}=1$ or $\rho(-q)_{s}=|G|_{s}$.
Hence if $\chi_{\lambda}(y)\neq0$, then $[\lambda]$ contains an
$(n-2)$-hook. In summary, if $\chi_{\lambda}(x)\chi_{\lambda}(y)\neq0$
then $[\lambda]$ contains both an $n$-hook and an $(n-2)$-hook.
This is only possible if $\lambda\in\{(n),(1^{n}),(n-1,1),(2,1^{n-2})\}$. 

The result for $n=4$ follows without the need to consider vanishing
on $y$ as the partitions listed in the statement are exactly those
containing a $4$-hook.
\end{proof}
\begin{prop}
\label{prop:x and y cover the group}Let $G=SU_{n}(q)$ where $n\geq4$
is even and $(n,q)\notin\{(4,2),(4,3),(6,2)\}$. Let $g\in G$ be
a non-semisimple element. Then $g\in x^{G}\cdot y^{G}$.
\end{prop}
\begin{proof}
We show that the normalised structure constant is non-zero by using
the formula given in Theorem \ref{thm:Burnsides formula}. By Lemma
\ref{lem:reduction to 4 character lemma} we thus far have a reduction
to

\begin{equation}
\kappa(x^{G},y^{G},g^{G})=\sum_{\chi_{\lambda}}\frac{\chi_{\lambda}(x)\chi_{\lambda}(y)\chi_{\lambda}(g^{-1})}{\chi_{\lambda}(1)}\label{eq:reduced sum for the normalised structure constant}
\end{equation}
 where the sum is over $\lambda\in\{(n),(1^{n}),(n-1,1),(2,1^{n-2})\}$.

Note that $\chi_{n}$ is the trivial character and $\chi_{1^{n}}$
is the Steinberg character $St$ which takes absolute value $1$ on
regular semisimple elements (\cite{Carter rep theory}, 6.4.7). Also
note the character degrees:
\[
St(1)=q^{\frac{n}{2}(n-1)},\ \chi_{(n-1,1)}(1)=\frac{q(q^{n-1}+1)}{q+1},\ \chi_{(2,1^{n-2})}(1)=\frac{q^{\frac{1}{2}(n-2)(n-1)}(q^{n-1}+1)}{q+1}.
\]
 Now $\chi_{(n-1,1)}$ is one of the so called irreducible Weil characters,
some characterisations of which are given by Tiep and Zalesskii in
\cite{Tiep Zalesskii weil reps}. Fix a generating element $\gamma$
of $\mathbb{F}_{q^{2}}^{*}$ and let $\delta=\gamma^{q-1}.$ Then
by (\cite{Tiep Zalesskii weil reps}, 4.1) we can compute the character
values using the formula
\begin{equation}
\chi_{(n-1,1)}(g)=\frac{(-1)^{n}}{q+1}\sum_{l=0}^{q}(-q)^{\text{dim Ker}(g-\delta^{-l})}.\label{eq:formula for the unipotent Weil rep}
\end{equation}
 By the choice of order, it follows that $x$ has no eigenvalues contained
in $\langle\delta\rangle$ and hence $\chi_{(n-1,1)}(x)=1$. Similarly,
$\chi_{(n-1,1)}(y)=-1$ as there exist two 1-dimensional eigenspaces
for the eigenvalues $\mu$ and $\mu^{-1}$. 

From (\cite{Tiep Zalesskii weil reps}, Sec. 4) we see that $\chi_{(n-1,1)}\in\mathcal{E}(G,(1))$
is the restriction of the unipotent character $\tilde{\chi}_{(n-1,1)}\in\text{Irr}(\tilde{G})$
and it follows from Lemmas \ref{lem dual charracters fixed on reg sem si-1}
and \ref{lem tiep and bob result} that $|\chi_{(n-1,1)}(x)|=|\tilde{\chi}_{(n-1,1)}(x)|=|\tilde{\chi}_{(2,1^{n-2})}(x)|$.
Then, as $C_{\boldsymbol{G}^{*}}(1)=\boldsymbol{G}^{*}$ is connected,
it follows from Section \ref{subsec:unitary groups prelims} that
$\tilde{\chi}_{(2,1^{n-2})}|_{G}=\chi_{(2,1^{n-1})}$. Hence $|\chi_{(n-1,1)}(x)|=|\chi_{(2,1^{n-2})}(x)|=1$.
It follows in an identical manner that $|\chi_{(2,1^{n-2})}(y)|=1$.

The value of the Steinberg character is given by (\cite{Carter rep theory},
6.4.7) : $St(z)=\pm|C_{G}(z)|_{p}$ if $z$ is semisimple and zero
otherwise. Hence the Steinberg summand vanishes in our consideration
of (\ref{eq:reduced sum for the normalised structure constant}).
Combining these results, we simplify (\ref{eq:reduced sum for the normalised structure constant})
to 
\[
\kappa(x^{G},y^{G},g^{G})=1+\frac{-(q+1)\chi_{(n-1,1)}(g^{-1})}{q(q^{n-1}+1)}\pm\frac{\chi_{(2,1^{n-2})}(g^{-1})(q+1)}{q^{\frac{1}{2}(n-2)(n-1)}(q^{n-1}+1)}.
\]
 As $g$ is non-scalar it cannot have an eigenvalue in $\langle\delta\rangle$
of multiplicity $n$ and thus $|\chi_{(n-1,1)}(g)|\leq\frac{1}{q+1}|(-q)^{n-1}-1|$.
We do not have a method for computing general character values of
$\chi_{(2,1^{n-2})}$ but it is enough to use the trivial bound $|\chi_{(2,1^{n-2})}(g)|\leq|C_{G}(g)|^{\frac{1}{2}}$.
Centralizer orders can be easily computed using the details in (\cite{Liebeck and Seitz book},
7.1 and \cite{Carter Centralizers paper}, Sec. 8) and this bound
is maximal when $g$ is unipotent with Jordan block structure $J_{2}\oplus J_{1}^{n-2}$.
Excluding a small number of exceptions when $n$ and $q$ are small,
we find that $\frac{|\chi_{(2,1^{n-2})}(g^{-1})|}{\chi_{(2,1^{n-2})}(1)}<\frac{1}{q^{2}}$.
In any case $\frac{|\chi_{(2,1^{n-2})}(g^{-1})|}{\chi_{(2,1^{n-2})}(1)}<\frac{q-1}{q}$
and hence
\[
\frac{(q+1)|\chi_{(n-1,1)}(g)|}{q(q^{n-1}+1)}+\frac{|\chi_{(2,1^{n-2})}(g^{-1})(q+1)|}{q^{\frac{1}{2}(n-2)(n-1)}(q^{n-1}+1)}<\frac{1}{q}+\frac{q-1}{q}=1.
\]
 We conclude that $\kappa(x^{G},y^{G},g^{G})\neq0$ and thus $g\in x^{G}\cdot y^{G}$.
\end{proof}
\textbf{Proof of Theorem }\ref{thm:inv wwidth of PSU} for $n$ even:

Let $G=SU_{n}(q)$ where $n\geq4$ is even. If $(n,q)\in\{(4,2),(4,3),(6,2)\}$
then the result can be checked using GAP \cite{GAP} so we assume
this is not the case. Let $g\in G$ be a semsimple element. Then by
Theorem \ref{thm:(Gow-)Theorem for ss elements}, it follows that
$\overline{g}\in x^{\overline{G}}y^{\overline{G}}$. Similarly, $\overline{g}\in x^{\overline{G}}y^{\overline{G}}$
for all non-semsimple elements $g\in G$ by Proposition \ref{prop:x and y cover the group}.
As $\overline{x}$ and $\overline{y}$ are strongly real by Proposition
\ref{prop:picking x and y}, it follows that iw$(\overline{G})\leq4$
and Theorem \ref{thm:inv wwidth of PSU} is proved for even $n\geq4$.

\subsection{$G=SU_{n}(q)$, n odd}

This section completes our work on the classical groups by considering
the simple unitary groups of odd dimension. More precisely we complete
the proof of Theorem \ref{thm:inv wwidth of PSU} for odd $n\geq3$.
Note that the initial cases $PSU_{3}(q)$ and $PSU_{5}(q)$ can be
dealt with easily using the computer package CHEVIE (\cite{CHEVIE}).We
do however begin with a slightly more detailed treatment of $SU_{3}(q)$
as knowledge of the involution width here will be of particular use
later in the analysis of exceptional Lie-type groups.
\begin{lem}
Let $G=SU_{3}(q)$ where $q\neq2$. Then $G$ has involution width
at most $4$.\label{lem:SU result for 3}
\end{lem}
\begin{proof}
Calculation of structure constants is straightforward in $SU_{3}(q)$
as the full character table is known (\cite{Simpson Frame character table of ssu3}).
Furthermore, products of conjugacy classes in $SU_{3}(q)$ have been
examined in the work of Orekovkov \cite{Orevkov}.

Firstly we note that $SU_{3}(q)$ contains a unique class of involutions
regardless of characteristic. If $3$ does not divide $q+1$, then
$PSU_{3}(q)=SU_{3}(q)$ and iw$(G)\leq4$ by (\cite{Orevkov}, 1.9).
Hence assume now that $3$ divides $q+1$ and $q\neq2$. Theorem 1.3
of \cite{Orevkov} details when exactly the identity element is contained
in a product of conjugacy classes, and denoting our class of involutions
by $A$, we find that for an arbitrary class $1\neq C\subset G$ either
$1\in A^{3}C$ or $1\in A^{4}C$. It follows that the involution width
of $SU_{3}(q)$ is at most 4 for $q\neq2$. 
\end{proof}
\begin{lem}
\label{lem:psu5}$PSU_{5}(q)$ has involution width at most 4.
\end{lem}
\begin{proof}
The generic character table of $PGU_{5}(q)$ is available in the computer
package CHEVIE \cite{CHEVIE}. These tables can be accessed with Maple
(\cite{Maple 15}) and we check the relevant structure constants computationally.
\end{proof}
The remainder of this section will be devoted to proving Theorem \ref{thm:inv wwidth of PSU}
where $n\geq7$ is odd. Throughout, let $G=SU_{n}(q)$. 

The proof of this this result will follow the same method as the case
when $n$ is even. That is, we first pick $x,y\in G$ such that the
projections of these elements are strongly real in $\overline{G}$
and then show that $G\backslash Z(G)\subset x^{G}\cdot y^{G}$. However
unlike before, when $n$ is odd we do not pick $x$ and $y$ from
weakly orthogonal tori. Thus when computing structure constants we
now have non-unipotent characters to contend with, which makes things
much more complicated, as we shall see.
\begin{prop}
\label{prop:picking x and y in n odd}There exist regular semisimple
elements $x\in T_{1,n-1}$ and $y\in T_{1^{3},n-3}$ such that $x$
and $y$ are strongly real in $G$.
\end{prop}
\begin{proof}
Firstly note that the torus $T_{1,n-1}$ of order $q^{n-1}-1$ is
cyclic and we can pick $x\in T_{1,n-1}$ of order $r$, where $r=$ppd$(q,n-1)$
when $n-1\equiv0$ mod $4$ and $r=$ppd$(q,\frac{n-1}{2})$ when
$n-1\equiv2$ mod $4$. These primes exist by Zsigmondy's Theorem.
Now $x$ has minimal polynomial of the form $p(t)=f(t)(t-1)$ where
$f(t)$ is a degree $n-1$ polynomial, irreducible over $\mathbb{F}_{q^{2}}$.
Furthermore, $f(t)$ is self-reciprocal, meaning that if $\alpha_{1},\dots,\alpha_{n-1}\in\overline{\mathbb{F}}_{q}$
are the roots of $f$, then $\prod(t-\alpha_{i})=f(t)=\prod(t-\alpha_{i}^{-1})$
. Clearly the elementary divisor $(t-1)$ is also self-reciprocal
and it follows that $x$ is real in $SU_{n}(q)$. It then follows
from (\cite{Vinroot real elements in SUn}, 7.1) and (\cite{Gates Singh and Vinroot Strongly real in unitary},
2.2 and 2.4) that $x$ is strongly real in $SU_{n}(q)$. 

A similar method is used for $y$ where we now pick a regular semisimple
element in the torus $T_{1^{3},n-3}$ of size $|T_{1^{3},n-3}|=(q^{n-3}-1)(q+1)^{2}$.
We take $y$ as the product of a regular element $y_{1}$ in the cyclic
torus $T_{n-3}<SU_{n-3}(q)$ and $y_{2}=$ diag$(\mu,\mu^{-1},1)\in T_{1^{3}}<SU_{3}(q)$.
In particular $y_{1}$ has order $r'$ where $r'=$ppd$(q,n-3)$ when
$n-3\equiv0$ mod $4$ and $r'=$ppd$(q,\frac{n-3}{2})$ when $n-3\equiv2$
mod $4$. Furthermore, if $\langle\delta\rangle\cong C_{q+1}\subset\mathbb{F}_{q^{2}}$
and $\epsilon=e^{2\pi i/(q+1)}$, then we choose $\mu=\delta^{a}$
such that $|1+2\text{Re}(\epsilon^{a})|\leq1$. This is always possible
and we can also ensure that $\mu\notin\{1,-1\}$. In fact if $q\equiv-1$
mod $3$ then we pick $\mu$ such that $1+2\text{Re}(\epsilon^{a})=0$.
This will be important for later calculations of structure constants.

Clearly $y_{2}$ has the torus $T_{1^{3}}$ as its centralizer in
$SU_{3}(q)$ and it follows that $C_{SU_{n}(q)}(y)=T_{1^{3},n-3}$
and so $y$ is regular. Note that $y$ has self-reciprocal minimal
polynomial $p(t)=f(t)(t-\mu)(t-\mu^{-1})(t-1)$ such that $f(t)$
is of degree $n-3$, self-reciprocal, and irreducible over $\mathbb{F}_{q^{2}}$.
Hence it follows as above that $y$ is strongly real in $SU_{n}(q)$.
\end{proof}

As we saw in Section \ref{subsec:n-even}, the bulk of the work is
now to show that the product $x^{G}\cdot y^{G}$ contains all non-semisimple
elements of $G$. Unlike before however, the two tori $T_{1,n-1}$
and $T_{1^{3},n-3}$ are not weakly orthogonal. Hence when computing
the structure constant it is possible that $\chi(x)\chi(y)\neq0$
for a far greater range of characters $\chi\in$Irr$(G)$. Exactly
when a character may be non-vanishing is detailed below.
\begin{prop}
\label{prop:non vanishing characters}Suppose $\chi\in$Irr$(G)$
such that $\chi(x)\chi(y)\neq0$. Then $\chi$ is one of the following
characters.
\end{prop}
\begin{enumerate}
\item \textit{Unipotent characters: }
\[
1_{G},\ St,\ \chi_{(n-3,2,1)},\ \chi_{(3,2,1^{n-5})}.
\]
\item \textit{Non-unipotent characters:} 
\[
\chi_{s^{t},(n-1)},\ \chi_{s^{t},(1^{n-1})},\ \chi_{s^{t},(n-2,1)},\ \chi_{s^{t},(2,1^{n-3})}.
\]
\end{enumerate}
\textit{Here} $s=\text{diag}(1,\dots,1,\delta)$ \textit{where} $\langle\delta\rangle\cong C_{q+1}\subset\mathbb{F}_{q^{2}}$,\textit{
and} $1\leq t\leq q$. \textit{Also} $C_{G^{*}}(s)\cong GU_{n-1}(q)$.
\begin{proof}
Applying Lemma \ref{lem:G and M lemma on the existence of a character in a non trivial Lusztig series}
to the elements $x$ and $y$, it follows that $\chi\in\mathcal{E}(G,(s))$
where $s\in T_{1,n-1}^{*}\cap T_{1^{3},n-3}^{*}.$ This intersection
consists of elements of the form $s=\text{diag}(1,\dots,1,\delta^{t})$
where $0\leq t\leq q$. First consider the case of possible unipotent
characters, that is when $t=0$ and so $s=1$. We proceed as in Lemma
\ref{lem:reduction to 4 character lemma}, proving that only the 4
unipotent characters $\chi_{\lambda}$ listed in the conclusion are
not of $r$- or $r'$-defect zero (where $r,$ $r'$ are as defined
in the proof of Proposition \ref{prop:picking x and y in n odd})
and then apply Theorem \ref{thm:Brauers theorem on prime defect}. 

Firstly assume that $\chi_{\lambda}$ is a unipotent character such
that $\chi_{\lambda}(x)\neq0.$ From the degree formula (\ref{eq:degree formula of unipotent characters}),
it is clear that $r$ divides $\chi_{\lambda}(1)=\rho_{\lambda}(-q)$
if and only if $\Phi_{n-1}(-q)$ divides $\rho_{\lambda}(-q)$. It
follows that either $\rho_{\lambda}(-q)_{r}=1$ or $\rho_{\lambda}(-q)_{r}=|G|_{r}$.
Then note that $\Phi_{n-1}(-q)$ divides $\rho_{\lambda}(-q)$ if
and only if $[\lambda]$ has no hooks of length $n-1$. Hence $\chi_{\lambda}(x)\neq0$
implies that $[\lambda]$ does indeed have an $(n-1)$-hook and so
the partition must be of the form $\lambda=(n),\,(1^{n})$ or $(k,2,1^{n-k-2})$
for some $k$. 

The case for characters that are possibly non-vanishing on $y$ is
similar as the order of $y$ is divisible by a primitive prime divisor
factor $r',$ of $\Phi_{n-3}(-q)$. As $n\geq7$, either $\rho_{\lambda}(-q)_{r'}=1$
or $\rho_{\lambda}(-q)_{r'}=|G|_{r'}$. Hence if $\chi_{\lambda}(y)\neq0$,
then $[\lambda]$ contains an $(n-3)$-hook. In summary, if $\chi_{\lambda}(x)\chi_{\lambda}(y)\neq0$
then $[\lambda]$ contains both an $(n-1)$-hook and an $(n-3)$-hook.
Namely $\lambda\in\{(n),(1^{n}),(n-3,2,1),(3,2,1^{n-5})\}$. 

To conclude, consider non-unipotent characters $\chi\in\mathcal{E}(G,(s^{t}))$
where $1\leq t\leq q$. Clearly $C=C_{G^{*}}(s)\cong GU_{n-1}(q)$
and thus $\chi=\chi_{s^{t},\lambda}$ where $\lambda\vdash n-1.$
As explained above, $\chi(x)\chi(y)\neq0$ implies that $[\lambda]$
contains both an $(n-1)$- and $(n-3)$-hook. Hence $\lambda\in\{(n-1),(1^{n-1}),(n-2,1),(2,1^{n-3})\}$
and we have the result.
\end{proof}
For the remainder of this section, denote the set of irreducible characters
listed in (1) and (2) of Proposition \ref{prop:non vanishing characters}
by $X\subset$Irr$(G)$.

\subsubsection{Values of characters in X\label{subsec:Values-of-characters}}

Before computing the structure constants in full, we compute the values
of each character $\chi\in X$ on the elements $x$ and $y$ defined
in Proposition \ref{prop:picking x and y in n odd}. Firstly, we have
$1_{G}$ and $St_{G}\in X$. These characters are well understood
and we have seen earlier that explicit formulae exist.

Next consider $\chi_{s^{t},(n-1)}$ and the dual $\chi_{s^{t},(1^{n-1})}$.
For ease of notation denote $F=\mathbb{F}_{q^{2}}$. The reducible
Weil character $\zeta_{n,q}$ of $GU_{n}(q)$ is defined by the formula
\begin{equation}
\zeta_{n,q}(g)=(-1)^{n}(-q)^{\text{dim Ker}_{\mathbb{F}}(g-1)}.\label{eq:Weil character values}
\end{equation}
This character is well studied in the literature (\cite{Tiep Zalesskii weil reps},
Sec. 4); in particular it has the following decomposition when restricted
to $G=SU_{n}(q)$:
\begin{equation}
\zeta_{n,q}|_{G}=\chi_{(n-1,1)}+\sum_{t=1}^{q}\chi_{s^{t},(n-1)}.\label{eq:weil decomp into erreds}
\end{equation}

Furthermore explicit formulae for the values of each $\chi_{s^{t},(n-1)}$
are known.
\begin{lem}
\label{lem:(REF-TIEP).-Formulae for Weil reps}(\cite{Tiep Zalesskii weil reps},
4.1). Let $\epsilon=e^{2\pi i/(q+1)}$ and $\langle\delta\rangle\cong C_{q+1}\subset\mathbb{F}_{q^{2}}$.
Then for $g\in G=SU_{n}(q)$,
\[
\chi_{s^{t},(n-1)}(g)=\frac{(-1)^{n}}{q+1}\sum_{l=0}^{q}\epsilon^{-tl}(-q)^{\text{dim Ker}_{\mathbb{F}}(g-\delta^{-l})}.
\]
\end{lem}
Far less is known about the dual characters $\chi_{s^{t},(1^{n-1})}$.
We can however apply Lemmas \ref{lem dual charracters fixed on reg sem si-1}
and \ref{lem:(REF-TIEP).-Formulae for Weil reps} to find $|\chi_{s^{t},(1^{n-1})}(x)|$
and $|\chi_{s^{t},(1^{n-1})}(y)|$. 

The second family of non-unipotent characters $\chi_{s^{t},(n-2,1)}$
has also been studied in the more recent literature. Now (\cite{Ore Conjecture paper},
6.6) describes any character $\chi\in\text{Irr}(SU_{n}(q))$ of degree
less than the order of $q^{3n-9}$ by associating $\chi$ to an $\alpha\in$Irr$(GU_{2}(q))$.
In particular we see that for a fixed $t$, $\chi_{s^{t},(n-2,1)}$
has labeling $D_{\alpha}^{\circ}$ where $\alpha=\chi_{q-1}^{(q+1,u)}$
for some $u\in\{1,\dots,q\}$. Here the labeling of Irr$(GU_{2}(q))$
is that used by Ennola (\cite{Ennola}). We then have by (\cite{Ore Conjecture paper},
5.5) that for $g\in SU_{n}(q)$, 
\begin{equation}
\chi_{s^{t},(n-2,1)}(g)=D_{\alpha}^{\circ}(g)=\frac{1}{|GU_{2}(q)|}\sum_{z\in GU_{2}(q)}\overline{\alpha(z)}\omega(zg)\label{eq:formula for D2}
\end{equation}
 where $\omega(h)=\zeta_{2n}(h)=(-q)^{\text{dim Ker}_{F}(h-1)}$ is
the reducible Weil character of $GU_{2n}(q)$ of degree $q^{2n}$
and $zg\in GU_{2}(q)\otimes GU_{n}(q)\subseteq GU_{2n}(q)$. \\
For the dual characters $\chi_{s^{t},(2,1^{n-3})}$, we will again
see that an application of Lemma \ref{lem dual charracters fixed on reg sem si-1}
or the trivial bound, $|\chi(g)|\leq|C_{G}(g)|^{\frac{1}{2}}$, will
be sufficient for calculating structure constants.
\begin{lem}
\label{lem:Level 1 and 2 c haracters on x and y}Let $x,y\in SU_{n}(q)$
be the regular semisimple elements defined in Proposition \ref{prop:picking x and y in n odd}.
Recall that $y$ has eigenvalue $\mu$ chosen such that $\mu=\delta^{a}$.
Then 
\[
\chi_{s^{t},(n-1)}(x)=\chi_{s^{t},(n-2,1)}(x)=1,
\]
\[
\chi_{s^{t},(n-1)}(y)=\chi_{s^{t},(n-2,1)}(y)=1+2\text{Re}(\epsilon^{a}),
\]
\[
|\chi_{s^{t},(1^{n-1})}(x)|=|\chi_{s^{t},(2,1^{n-2})}(x)|=1,
\]
\[
|\chi_{s^{t},(1^{n-1})}(y)|=|\chi_{s^{t},(2,1^{n-2})}(y)|=|1+2\text{Re}(\epsilon^{a})|.
\]
\end{lem}
\begin{proof}
The values of $\chi_{s^{t},(n-1)}$ follow from Lemma \ref{lem:(REF-TIEP).-Formulae for Weil reps}.
The values of $\alpha=\chi_{q-1}^{(q+1,u)}\in\text{Irr}(GU_{2}(q))$
are given in full in (\cite{Ennola}) and we then use (\ref{eq:formula for D2})
to evaluate $D_{\alpha}^{\circ}=\chi_{s^{t},(n-2,1)}$. 

To evaluate the characters $\chi_{s^{t},(1^{n-1})}$ and $\chi_{s^{t},(2,1^{n-2})}$
we follow the method given in the proof of Proposition \ref{prop:x and y cover the group}:
from (\cite{Tiep Zalesskii weil reps}, Sec. 4) we see that $\chi_{s^{t},(n-1)}$
is the restriction of the irreducible Weil character $\tilde{\chi}_{\tilde{s}^{t},(n-1)}\in\text{Irr}(\tilde{G})$
where $\tilde{s}$ is a fixed preimage of $s^{t}\in G^{*}$. It follows
from Lemmas \ref{lem dual charracters fixed on reg sem si-1} and
\ref{lem tiep and bob result} that $\text{|}\chi_{s^{t},(n-1)}(x)|=\text{|}\text{\ensuremath{\tilde{\chi}}}_{\tilde{s}^{t},(n-1)}(x)|=|\tilde{\chi}_{\tilde{s}^{t},(1^{n-1})}(x)|$.
Then, as $C_{\boldsymbol{G}^{*}}(s^{t})\cong GU_{n-1}(\overline{\mathbb{F}}_{q})$
is connected, it follows from Section \ref{subsec:unitary groups prelims}
that $\tilde{\chi}_{s^{t},(1^{n-1})}|_{G}=\chi_{s^{t},(1^{n-1})}$.
Hence $|\chi_{s^{t},(n-1)}(x)|=|\chi_{s^{t},(1^{n-1})}(x)|=1$. The
value of $|\chi_{s^{t},(1^{n-1})}(y)|$ follows identically. 

To evaluate $\text{|}\chi_{s^{t},(2,1^{n-2})}(x)|$, first note from
formula (\ref{eq:formula for D2}) that $\chi_{s^{t},(n-2,1)}$ is
$\tilde{G}$-invariant. Hence as $\tilde{G}/G$ is cyclic, $\chi_{s^{t},(n-2,1)}$
extends to a character of $\tilde{G}$ by Lemma \ref{lem:cyclic extensions and extendible characters}.
From Sec \ref{subsec:unitary groups prelims} such extensions lie
in Lusztig series $\mathcal{E}(\tilde{G},\tilde{s}z)$, where $\tilde{s}$
is a fixed preimage of $s^{t}\in G^{*}$ and $z\in Z(G)$. Note that
if $\tilde{\chi}\in\mathcal{E}(\tilde{G},\tilde{s}z)$ restricts irreducibly
to $\chi_{s^{t},(n-2,1)}$, then $\tilde{\chi}(1)_{p}=\chi_{s^{t},(n-2,1)}(1)_{p}=q.$
Checking degrees of the unipotent characters of $C_{\tilde{G}}(\tilde{s}z)\cong GU_{n-1}(q)\times GU_{1}(q)$,
it follows that $\tilde{\chi}=\tilde{\chi}_{\tilde{s}z,(n-2,1)}$.
Noting again that $C_{\boldsymbol{G}^{*}}(s^{t})$ is connected, the
proof now follows exaclty as above, by considering the Alvis-Curtis
dual of $\tilde{\chi}$. 
\end{proof}
Note that when we apply these values later in the calculation of structure
constants, we will only require the bound $|1+2\text{Re}(\epsilon^{a})|\leq1$
given in the proof of Proposition \ref{prop:picking x and y in n odd}.
However recall that when $q=2$, $y$ is chosen such that $1+2\text{Re}(\epsilon^{a})=0.$ 

Lastly we want to estimate values of the unipotent character $\chi_{(n-3,2,1)}$.
We use the theory of dual pairs to extend the work in (\cite{Ore Conjecture paper},
Sec. 6.1).

Let $G=SU_{n}(q)$ with $n\geq7$ and odd and let $S:=GU_{3}(q)$.
Take $G=SU(W)$ where $W=\langle v_{1},\dots,,v_{n}\rangle_{\mathbb{F}_{q^{2}}}$
is endowed with a Hermitian form $(\cdot,\cdot),$ with Gram matrix
diag$(1,\dots,1)$ with respect to the basis $v_{1},\dots,v_{n}$.
Similarly we view $S$ as $GU(U)$ where $U=\langle e_{1},e_{2},e_{3}\rangle_{\mathbb{F}_{q^{2}}}$is
endowed with Hermitian form $(\cdot,\cdot),$ with Gram matrix diag$(1,1,1)$
in the basis $e_{1},e_{2},e_{3}$. Next we consider $V=U\otimes W$
with the Hermitian form $(\cdot,\cdot)$ defined via $(u\otimes w,u'\otimes w')=(u,u')\cdot(w,w')$
for $u,u'\in U$ and $w,w'\in W$. The action of $S\times G$ on $V$
then induces a natural homomorphism $S\times G\rightarrow\Gamma:=GU(V)=GU_{3n}(q).$ 

Let $\omega=\zeta_{3n,q}$ denote the reducible Weil character of
$\Gamma$, as defined by (\ref{eq:Weil character values}). It is
by studying the restriction of $\omega$ to $S\times G$ that we will
find the character $\chi_{(n-3,2,1)}$. 
\begin{lem}
Let $n\geq7$. \label{lem:(w|g,w|G)}Then 
\[
(\omega|_{G},\omega|_{G})_{G}=(q+1)(q^{3}+1)(q^{5}+1).
\]
 Further $(\omega|_{G},1_{G})_{G}=0$ and $(\omega|_{G},\zeta_{n,q})_{G}=(q+1)(q^{3}+1).$ 
\end{lem}
\begin{proof}
Let $A$ be the matrix of $g\in G$ in the basis $v_{1},\dots,v_{n}$
of $W$. Then $g$ has matrix diag$(A,A,A)$ in the basis $\{e_{1}\otimes v_{i},e_{2}\otimes v_{i},e_{3}\otimes v_{i}\}$
of $V$. It also follows that $\omega|_{G}=(\zeta_{n,q})^{3}$ where
$\zeta_{n,q}$ denotes the reducible Weil character of $GU_{n}(q)$
as in (\ref{eq:Weil character values}). Furthermore
\begin{eqnarray*}
(\omega|_{G},\omega|_{G})_{G} & = & \frac{1}{|G|}\sum_{A\in G}\omega(A)^{2}\\
 & = & \frac{1}{|G|}\sum_{A\in G}\text{Fix}(A)^{3},
\end{eqnarray*}
 where Fix$(A)$ denotes the number of fixed points under the action
on the natural module $W$. But this is exactly the number of of $G$-orbits
on $W\times W\times W.$ Using Witt's lemma and the assumptions on
$n$ we find that the number of $G$-orbits is exactly $(q+1)(q^{3}+1)(q^{5}+1)$.\\
Next note that $(\omega|_{G},1_{G})_{G}=(\zeta_{n,q}{}^{2},\zeta_{n,q})_{G}$
and similarly $(\omega|_{G},\zeta_{n,q})_{G}=(\zeta_{n,q}{}^{2},\zeta_{n,q}{}^{2})_{G}$.
The decomposition of $\zeta_{n,q}{}^{2}$ is studied in (\cite{Ore Conjecture paper},
6.1) and we can check that indeed $(\omega|_{G},1_{G})_{G}=0$ and
$(\omega|_{G},\zeta_{n,q})_{G}=(q+1)(q^{3}+1).$
\end{proof}
\begin{prop}
\label{prop:classsifying the level 3 characters}Let $S=GU_{3}(q)$
and $G=SU_{n}(q)$ with $n\geq7$ odd. Then the restriction $\zeta_{3n,q}|_{S\times G}$
decomposes as $\sum_{\alpha\in\text{Irr}(S)}\alpha\otimes D_{\alpha}$
where $D_{\alpha}$ is a character of $G$. Define $a_{\alpha}=(D_{\alpha},\chi_{(n-1,1)})_{G}$,
$b_{\alpha}^{t}=(D_{\alpha},\chi_{s^{t},(n-1)})_{G}$, for $1\leq t\leq q,$
and 
\[
D_{\alpha}^{\circ}=D_{\alpha}-a_{\alpha}\chi_{(n-1,1)}-\sum_{t=1}^{q}b_{\alpha}^{t}\chi_{s^{t},(n-1)}.
\]
Then the characters $D_{\alpha}^{\circ}$ of $G=SU_{n}(q)$ are all
irreducible and distinct.
\end{prop}
\begin{proof}
Applying (\cite{Ore Conjecture paper}, 5.5) to the character $\omega=\zeta_{3n,q}$
gives the decomposition $\omega|_{S\times G}=\sum_{\alpha\in\text{Irr}(S)}\alpha\otimes D_{\alpha}$
such that 
\begin{equation}
D_{\alpha}(g)=\frac{1}{|S|}\sum_{x\in S}\overline{\alpha(x)}\omega(xg).\label{eq:Formula for D_alpha}
\end{equation}
By definition, 
\[
\omega|_{G}=\sum_{\alpha\in\text{Irr}(S)}\alpha(1)\left(a_{\alpha}\chi_{(n-1,1)}+\sum_{t=1}^{q}b_{\alpha}^{t}\chi_{s^{t},(n-1)}\right)+\sum_{\alpha\in\text{Irr}(S)}\alpha(1)D_{\alpha}^{\circ}
\]
 and it follows from Lemma \ref{lem:(w|g,w|G)} and (\ref{eq:weil decomp into erreds})
that 
\begin{equation}
\sum_{\alpha\in\text{Irr}(S)}\alpha(1)\left(a_{\alpha}+\sum_{t=1}^{q}b_{\alpha}^{t}\right)=(q^{3}+1)(q+1).\label{eq:sum of non squares}
\end{equation}
For ease of notation, relabel $\sum_{\alpha\in\text{Irr}(S)}\alpha(1)a_{\alpha}=a$
and $\sum_{\alpha\in\text{Irr}(S)}\alpha(1)b_{\alpha}^{t}=b_{t}$.\\
It also follows from Lemma \ref{lem:(w|g,w|G)} that 
\[
\sum_{\alpha\in\text{Irr}(S)}\alpha(1)^{2}=|S|=q^{3}(q^{3}+1)(q^{2}-1)(q+1)=(\omega|_{G},\omega|_{G})_{G}-(q+1)(q^{3}+1)^{2}.
\]
Hence
\[
\sum_{\alpha\in\text{Irr}(S)}\alpha(1)^{2}=(\sum_{\alpha\in\text{Irr}(S)}\alpha(1)D_{\alpha}^{\circ},\sum_{\alpha\in\text{Irr}(S)}\alpha(1)D_{\alpha}^{\circ})_{G}+(a^{2}+\sum_{t=1}^{q}b_{t}^{2})-(q+1)(q^{3}+1)^{2}.
\]
Applying the Cauchy-Schwarz inequality to (\ref{eq:sum of non squares})
yields 
\[
(a^{2}+\sum_{t=1}^{q}b_{t}^{2})\geq(q+1)(q^{3}+1)^{2}.
\]
 Thus if each $D_{\alpha}^{\circ},$ $\alpha\in$Irr$(S)$ has positive
degree, it will follow that the characters are irreducible and distinct.
The character table of $S$ is known and we follow the notation of
\cite{Ennola}. In particular there are 8 families of irreducible
characters of degrees $1,\,q^{2}-q,\,q^{3},\,q^{2}-q+1,\,q(q^{2}-q+1),\,(q-1)(q^{2}-q+1),\,q^{3}+1$
and $(q+1)(q^{2}-1).$ 

We compute the corresponding $D_{\alpha}(1)$ using (\ref{eq:Formula for D_alpha})
and it is a straightforward check that $D_{\alpha}(1)>q^{n}(q^{3}+1)(q+1)/\alpha(1)$
for all $\alpha\in\text{Irr}(S)$. Furthermore, by the definition
of $D_{\alpha}^{\circ}$, 
\begin{equation}
D_{\alpha}(1)-D_{\alpha}^{\circ}(1)\leq(\omega|_{G},\zeta_{n,q})_{G}\zeta_{n,q}(1)/\alpha(1)=q^{n}(q+1)(q^{3}+1)/\alpha(1).\label{eq:max character degree difference}
\end{equation}
 Hence $D_{\alpha}^{\circ}(1)>0$ for all $\alpha\in\text{Irr}(S)$
and the proof is finished. 
\end{proof}
\begin{rem*}
By the proof of Proposition \ref{prop:classsifying the level 3 characters},
$(a^{2}+\sum_{t=1}^{q}b_{t}^{2})=(q+1)(q^{3}+1)^{2}$ and it follows
from Cauchy-Schwarz that $a=b_{t}=(q^{3}+1)$ for $1\leq t\leq q$.
Hence 
\[
\omega|_{G}=(q^{3}+1)\zeta_{n,q}+\sum_{\alpha\in\text{Irr}(S)}\alpha(1)D_{\alpha}^{\circ}.
\]
 The degrees $D_{\alpha}(1)$ are listed in Table \ref{table:1}.
The notation for characters $\alpha\in\text{Irr}(GU_{3}(q))$ is taken
from \cite{Ennola}.
\end{rem*}

\begin{table} \begin{center} \caption{Degrees of $D_{\alpha}$ for $G=SU_{n}(q)$} \label{table:1} \end{center} \begin{adjustwidth}{-2cm}{} \renewcommand{\arraystretch}{2} \begin{tabular}{cc} \hline  $\alpha\in\text{Irr}(GU_{3}(q))$ & $D_{\alpha}(1)$\tabularnewline \hline  $\chi_{1}^{(q+1)}$  & $\frac{q^{3}(q^{n}+1)(q^{n-1}-1)(q^{n-5}-1)}{(q^{3}+1)(q^{2}-1)(q+1)}+\frac{q(q^{n-1}-1)}{(q+1)}$\tabularnewline $\chi_{1}^{(t)}$, $1\leq t\leq q$ & $\frac{(q^{n}+1)(q^{n-1}-1)(q^{n-2}+1)}{(q^{3}+1)(q^{2}-1)(q+1)}.$\tabularnewline $\chi_{q^{3}}^{(q+1)}$  & $\frac{q^{6}(q^{n-1}-1)(q^{n-2}+1)(q^{n-4}+1)}{(q^{3}+1)(q^{2}-1)(q+1)}+\frac{q(q^{n-1}-1)}{(q+1)}$\tabularnewline $\chi_{q^{3}}^{(t)}$, $1\leq t\leq q$ & $\frac{q^{3}(q^{n}+1)(q^{n-1}-1)(q^{n-2}+1)}{(q^{3}+1)(q^{2}-1)(q+1)}$\tabularnewline $\chi_{q^{2}-q}^{(q+1)}$ & $\frac{q^{4}(q^{n}+1)(q^{n-2}+1)(q^{n-4}+1)}{(q^{3}+1)(q+1)^{2}}$\tabularnewline $\chi_{q^{2}-q}^{(t)}$, $1\leq t\leq q$  & $\frac{q(q^{n}+1)(q^{n-1}-1)(q^{n-2}+1)}{(q^{3}+1)(q+1)^{2}}$\tabularnewline $\chi_{q^{2}-q+1}^{(t,q+1)}$, $1\leq t\leq q$  & $\frac{q^{2}(q^{n}+1)(q^{n-1}-1)(q^{n-4}+1)}{(q^{2}-1)(q+1)^{2}}+\frac{(q^{n}+1)}{(q+1)}$\tabularnewline $\chi_{q^{2}-q+1}^{(q+1,u)}$, $1\leq u\leq q$ & $\frac{q(q^{n}+1)(q^{n-1}-1)(q^{n-3}-1)}{(q^{2}-1)(q+1)^{2}}$\tabularnewline $\chi_{q^{2}-q+1}^{(t,u)}$, $1\leq t \neq u\leq q$   & $\frac{(q^{n}+1)(q^{n-1}-1)(q^{n-2}+1)}{(q^{2}-1)(q+1)^{2}}$\tabularnewline $\chi_{q(q^{2}-q+1)}^{(t,q+1)}$, $1\leq t\leq q$ & $\frac{q^{3}(q^{n}+1)(q^{n-1}-1)(q^{n-3}-1)}{(q^{2}-1)(q+1)^{2}}+\frac{(q^{n}+1)}{(q+1)}$\tabularnewline $\chi_{q(q^{2}-q+1)}^{(q+1,u)}$, $1\leq u\leq q$  & $\frac{q^{2}(q^{n}+1)(q^{n-1}-1)(q^{n-3}-1)}{(q^{2}-1)(q+1)^{2}}$\tabularnewline $\chi_{q(q^{2}-q+1)}^{(t,u)}$, $1\leq t\neq u\leq q$  & $\frac{q(q^{n}+1)(q^{n-1}-1)(q^{n-2}+1)}{(q^{2}-1)(q+1)^{2}}$\tabularnewline $\chi_{(q-1)(q^{2}-q+1)}^{(t,u,q+1)}$, $1\leq t<u<q+1$ & $\frac{q(q^{n}+1)(q^{n-1}-1)(q^{n-3}-1)}{(q+1)^{3}}$ \tabularnewline $\chi_{(q-1)(q^{2}-q+1)}^{(t,u,v)}$, $1\leq t<u<v<q+1$ & $\frac{(q^{n}+1)(q^{n-1}-1)(q^{n-2}+1)}{(q+1)^{3}}$\tabularnewline $\chi_{q^{3}+1}^{(q+1,u)}$, $1\leq u\leq q^{2}-2$ & $\frac{q(q^{n}+1)(q^{n-1}-1)(q^{n-3}-1)}{(q^{2}-1)(q+1)}$ \\ $u\not\equiv0$ mod $(q-1)$ & \tabularnewline $\chi_{q^{3}+1}^{(t,u)}$, $1\leq t\leq q$ , $1\leq u\leq q^{2}-2$ & $\frac{(q^{n}+1)(q^{n-1}-1)(q^{n-2}+1)}{(q^{2}-1)(q+1)}$ \\  $u\not\equiv0$ mod $(q-1)$ & \tabularnewline $\chi_{(q+1)(q^{2}-1)}^{(t)}$,  $t\not\equiv0$ mod $(q^{2}-q+1)$; & $\frac{(q^{n}+1)(q^{n-1}-1)(q^{n-2}+1)}{(q^{3}+1)}$ \\ if $t_{1}\equiv tq^{2}$ or $t_{2}\equiv tq^{4}$ mod $(q^{3}+1)$, then $\chi^{(t)}=\chi^{(t_{1})}=\chi^{(t_{2})}$ &  \tabularnewline \hline  \end{tabular} \end{adjustwidth} \end{table}

Recall that we wish to find a formula for the values of the unipotent
character $\chi_{(n-3,2,1)}$. From Table \ref{table:1}, we see that
there is precisely one character $D_{\alpha}$ of the correct degree.
Specifically, when $\alpha=\chi_{q^{2}-q}^{(q+1)}$, 
\[
D_{\alpha}(1)=\frac{q^{4}(q^{n}+1)(q^{n-2}+1)(q^{n-4}+1)}{(q^{3}+1)(q+1)^{2}}=\chi_{(n-3,2,1)}(1).
\]

\begin{prop}
\label{prop: The D_alpha is the unipotent we're looking for}Let $\alpha=\chi_{q^{2}-q}^{(q+1)}\in\text{Irr}(GU_{3}(q))$.
Then $D_{\alpha}=D_{\alpha}^{\circ}=\chi_{(n-3,2,1)}$. 
\end{prop}
From this we can then calculate values of $\chi_{(n-3,2,1)}$ using
the known character table of $GU_{3}(q)$ and the formula (\ref{eq:Formula for D_alpha}).
Define 
\[
d:=\frac{q^{n}(q+1)(q^{3}+1)}{q(q-1)}.
\]
To prove the proposition we first note from (\ref{eq:max character degree difference})
in the proof of Proposition \ref{prop:classsifying the level 3 characters},
that 
\[
D_{\alpha}(1)-D_{\alpha}^{\circ}(1)\leq d.
\]
It therefore suffices to show that $\chi_{(n-3,2,1)}$ is the only
irreducible character with degree in the range 
\[
D_{\alpha}(1)-d\leq\chi(1)\leq D_{\alpha}(1).
\]

We first consider the slightly easier case of unipotent characters,
where we will need the following result.
\begin{lem}
(\cite{Tiep and Zalesskii Minimal characters}, 2.1). Let $2\leq a_{1}<a_{2}<\dots<a_{l}$
be integers, $\epsilon_{1},\dots,\epsilon_{l}\in\{1,-1\}.$ Th\label{lem:Lemma on partition powers}en
\[
\frac{1}{2}<\frac{(q^{a_{1}}+\epsilon_{1})\dots(q^{a_{l}}+\epsilon_{l})}{q^{a_{1}+\dots+a_{l}}}<2.
\]
\end{lem}

\begin{lem}
Let $\chi_{\mu}\in\text{Irr}(G)$ and $|\chi_{\mu}(1)-\chi_{(n-3,2,1)}(1)|\le d$.
Then $\mu=(n-3,2,1)$.\label{prop: degree of unipotent characterss and proximity to D3}
\end{lem}
\begin{proof}
For ease we adopt the notation of \cite{Tiep and Zalesskii Minimal characters}
in this proof, reversing the order in which partitions are written. 

Let $\chi_{\mu}$ be the unipotent character corresponding to the
partition $(\mu_{1},\dots,\mu_{m})\vdash n$ and set $\lambda_{i}=\mu_{i}+i-1$
for all $i$. The degree formula (\ref{eq:degree formula of unipotent characters})
in Section \ref{sec:Unitary-Groups} can be rewritten (\cite{Tiep and Zalesskii Minimal characters},
4.2A) in terms of $\lambda_{i}$ to give 
\begin{equation}
\chi_{\mu}(1)=\frac{(q+1)(q^{2}-1)\dots(q^{n}-(-1)^{n})\prod_{i'<i}(q^{\lambda_{i}}-(-1)^{\lambda_{i}+\lambda'_{i}}q^{\lambda'_{i}})}{\left(q^{\binom{m-1}{2}+\binom{m-2}{2}+\dots}\right)\prod_{i}\prod_{k=1}^{\lambda_{i}}(q^{k}-(-1)^{k})}.\label{eq:new form for degree formula}
\end{equation}
To prove the Lemma we shall use this form of the degree to show that
if $\mu\neq(1,2,n-3)$ then $|\chi_{\mu}(1)-\chi_{(1,2,n-3)}(1)|>d$.
Firstly, the cases $n=7,9$ can be checked directly so we may assume
that $n\geq11$. Furthermore we can check explicitly that the statement
holds when $\mu_{m}\geq n-3$ and hence we assume that $\mu_{m}\leq n-4$.
We show that $\chi_{\mu}(1)\geq q^{4n-17}$, using induction on the
length of the partition, $m$. The conclusion will then follow as
$q^{4n-17}-\chi_{(1,2,n-3)}(1)>d$ for $n\geq9$.

The base case for the induction is the set of two-part partitions,
namely $\mu=(k,n-k)$ where $k\geq4$. When $k=4,$ 
\[
\chi_{(4,n-4)}(1)=\frac{q^{4}(q^{n}+1)(q^{n-1}-1)(q^{n-2}+1)(q^{n-7}-1)}{(q^{4}-1)(q^{3}+1)(q^{2}-1)(q+1)}
\]
 and it is an easy check that this is greater than $q^{4n-17}$. When
$k\geq5$, $\chi_{\mu}(1)$ is 
\[
\frac{(q+1)(q^{2}-1)\dots(q^{n}+1)(q^{n-k+1}+q^{k})}{(q+1)(q^{2}-1)\dots(q-(-1)^{k})\cdot(q+1)(q^{2}-1)\dots(q^{n-k+1}-(-1)^{n-k+1})}
\]
\[
=\frac{(q^{n-k+2}-(-1)^{n-k+2})\dots(q^{n}+1)}{(q^{2}-1)\dots(q-(-1)^{k})}\cdot\frac{q^{k}(q^{n-2k+1}-1)}{q+1}
\]
 
\[
>q^{k}\frac{(q^{n-k+2}\cdot q^{n-k+3}\cdot\dots\cdot q^{n})/2}{2q^{2}\cdot q^{3}\cdot\dots\cdot q^{k}}\geq q^{(k-2)(n-k)+n-2}.
\]
 Here we have used Lemma \ref{lem:Lemma on partition powers}. As
$k\geq5$ it follows that 
\[
(k-2)(n-k)+(n-2)-(4n-17)=(n-k)(k-5)+3(k+5)>0
\]
 which gives the conclusion for two-part partitions.

For the induction step assume that $\mu$ has length $m\geq3$. Denote
$\mu_{1}=k$ and note that the conditions $m\geq3$ and $n\geq11$
imply $k\leq n-8.$ If $\mu$ is $(2^{2},n-4),\,(1,3,n-4),\,(2,3,n-5),\,(3^{2},n-6),\,(1^{2},2,n-4)$
or $(1^{4},n-4)$ we can check directly that $\chi_{\mu}(1)>q^{4n-17}$.
Hence we assume that $\mu$ is not of this form, in addition to the
original assumption that $\mu_{m}\leq n-4.$ 

Consider the partition $\nu\vdash n-k$ of length $m-1$ defined by
$\nu_{i}=\mu_{i+1}$. By the above exclusions, it follows that $\nu_{m-1}\leq n-k-4$
and therefore $\chi_{\nu}(1)>q^{4(n-k)-17}$ by the induction hypothesis.
From (\ref{eq:new form for degree formula}) it also follows that
\[
\frac{\chi_{\mu}(1)}{\chi_{\nu}(1)}=\frac{\prod_{i=n-k+1}^{n}(q^{i}-(-1)^{i})\cdot\prod_{i=2}^{m}(q^{\lambda_{i}}-(-1)^{\lambda_{i}+k}q^{k})}{\prod_{i=1}^{n}(q^{i}-(-1)^{i})\cdot\prod_{i=2}^{m}(q^{\lambda_{i}}-(-1)^{\lambda_{i}})}.
\]
 Note that 
\[
\lambda_{m}=\mu_{m}+m-1\leq\mu_{m}+\mu_{1}+\mu_{2}+\dots+\mu_{m-1}=n
\]
 and hence 
\[
2\leq\lambda_{3}-k\leq\dots\leq\lambda_{m}-k\leq n-k.
\]
 We can therefore apply Lemma \ref{lem:Lemma on partition powers}
to the collections $(a_{1},\dots,a_{l})=(\lambda_{3}-k,\dots,\lambda_{m}-k,n-k+1,\dots,n)$
and $(a_{1},\dots,a_{l})=(2,3,\dots,k,\lambda_{2},\dots,\lambda_{m})$.
Also using the inequality
\[
(q^{\lambda_{2}}-(-1)^{\lambda_{2}+k}q^{k})/(q+1)\geq q^{\lambda_{2}-1}/2,
\]
 we have 
\[
\frac{\chi_{\mu}(1)}{\chi_{\nu}(1)}=\frac{\frac{1}{2}q^{\lambda_{2}-1}\cdot\frac{1}{2}q^{\lambda_{3}}\cdot\dots\cdot q^{\lambda_{m}}\cdot q^{n-k+1}\cdot\dots\cdot q^{n}}{2q^{2}\cdot\dots\cdot q^{k}\cdot q^{\lambda_{2}}\cdot\dots\cdot q^{\lambda_{n}}}=\frac{q^{k(n-k)}}{8}\geq q^{k(n-k)-3}.
\]
 Thus $\chi_{1,\mu}(1)>q^{(4+k)(n-k)-20}=q^{(4n-17)+(k(n-k)-4k-3)}>q^{4n-17}$
as $n-k\geq8$ and the induction is complete. 
\end{proof}
We next move on to proving the analogous result for non-unipotent
characters. 

\begin{lem}
Let $\chi\in$Irr$(SU_{n}(q))$ be a n\label{prop:Non-unipotent degrees close to my character}on-unipotent
character such that $|\chi(1)-\chi_{(n-3,2,1)}(1)|\le d$. Then $(n,q)=(7,2)$
and
\[
\chi(1)\in\{6622,10\,234,9030\}.
\]
 
\end{lem}
\begin{proof}
Let $\chi=\chi_{s,\overline{\lambda}}$ where $1\neq s\in G^{*}$
as defined in Section \ref{subsec:unitary groups prelims}. Following
(\cite{Tiep and Zalesskii Minimal characters}, 4.1), let $\tilde{G}=GU_{n}(q)$
and consider an inverse image $\tilde{s}\in\tilde{G}$ of $s$. Let
$C$ be the complete inverse image of $C^{*}=C_{G^{*}}(s)$ in $\tilde{G}$
and set $D=C_{\tilde{G}}(\tilde{s}).$ Note that $(G^{*}:C^{*})=(\tilde{G}:C)$
and $D$ is a normal subgroup of $C$ such that $C/D\hookrightarrow C_{q+1}$.
Thus by Proposition \ref{prop:(Jordan-decomposition-of characters and degree formula},
\[
\chi(1)\geq(G^{*}:C^{*})_{p'}\geq(\tilde{G}:D)_{p'}/(q+1).
\]
 Now $\tilde{G}$ acts on the natural module $V=\mathbb{F}_{q^{2}}^{n}$
and we denote the characteristic polynomial of $\tilde{s}$ in this
action by $P(t)\in\mathbb{F}_{q^{2}}[t].$ Then $P(t)$ is a product
$\prod_{i=1}^{l}f_{i}(t)^{m_{i}}$ of distinct irreducible polynomials
$f_{i}(t)$ over $\mathbb{F}_{q^{2}}$, $1\leq i\leq l$. 

Firstly consider the case where $l=1$. It follows from the non-triviality
of $s$ that $f=f_{1}$ has degree $k\geq2$. If $\lambda\in\overline{\mathbb{F}}_{q}$
is a root of $f(t)$, then $\mathbb{\mathbb{F}}_{q^{2}}(\lambda)=\mathbb{F}_{q^{2k}}$
and the roots $\lambda_{j}=\lambda^{q^{2(j-1)}}$ of $f(t)$ are distinct
for $1\leq j\leq k$. From (\cite{Tiep and Zalesskii Minimal characters},
4.2B) we find further that $k$ is odd and that $D=GU_{m}(q^{k})$
such that $mk=n$.\\
Thus by Lemma \ref{lem:Lemma on partition powers}
\begin{eqnarray*}
\chi(1) & \geq & \frac{(q+1)(q^{2}-1)\dots(q^{mk}-(-1)^{mk})}{(q+1)(q^{k}+1)\dots(q^{mk}-(-1)^{m})}\\
 & > & \frac{\frac{1}{2}q^{2+3+\dots+mk}}{2q^{k+2k+\dots+mk}}=\frac{q^{\frac{1}{2}m^{2}k(k-1)-1}}{4}.
\end{eqnarray*}
Note that $\chi_{(n-3,2,1)}(1)+d<q^{3n-7}$ for $n\geq7$. Thus it
suffices to show that $q^{\frac{1}{2}m^{2}k(k-1)-3}\geq q^{3n-7}.$
As $k\geq3$, $\frac{m^{2}k(k-1)}{2}=n\cdot\frac{n-m}{2}\geq\frac{n^{2}}{3}$
and then $\frac{n^{2}}{3}\geq3n-4$ when $n\geq9.$ When $n=7$, it
follows that $k=7$ and we still have $\frac{m^{2}k(k-1)}{2}-3\geq3n-7$.
This concludes the case where $l=1$ and hence we assume that $l>1$
from now on.

Here we see that $D$ preserves a non-trivial decomposition $V=\bigoplus_{i=1}^{l}V_{i}$
where $V_{i}=P_{i}(s)(V)$ and $P_{i}(t)=P(t)/f_{i}(t)^{m_{i}}$.
We denote $n_{i}=\dim V_{i}=m_{i}\deg(f_{i})$ and without loss of
generality we assume that $\dim V_{1}\leq\dots\leq\dim V_{l}$.

Suppose $D$ fixes a non-zero totally isotropic subspace $W$ in $V.$
Denote $\dim W=k$ and $b_{k}=(\tilde{G}:\tilde{G}_{W})_{p'}.$ Clearly
$1\leq k\leq\frac{n-1}{2}$, and we check using Lemma \ref{lem:Lemma on partition powers}
that if $k\geq2$ then $b_{k}\geq\text{min}\{b_{2},b_{(n-1)/2}\}$.
Hence if $k\geq2,$ then we have 
\[
\chi(1)\geq\text{min}\{b_{2},b_{(n-1)/2}\}>\chi_{(n-3,2,1)}(1)+d.
\]
We may therefore assume that any $D$-invariant totally isotropic
subspace has dimension 1. 

Similarly, if $D$ preserves a degenerate subspace $W\subset V$,
we may assume that $W$ is totally isotropic and of dimension 1: firstly
note that if $U=W\cap W^{\perp}$, then $\dim U=1$ by the above.
Let $e$ denote an isotropic basis vector for $U$ and let $v_{j}\in V$
such that $e,v_{1},\dots,v_{k}$ is a basis for $W$, where $\langle v_{1},\dots,v_{k}\rangle$
is non-degenerate. Note that $D\leq C_{\tilde{G}}(W)\le C_{\tilde{G}}(\langle e\rangle)=P_{1}$,
where $P_{1}$ is maximal parabolic. As $C_{\tilde{G}}(W)$ stabilises
the non-degenerate $k$-space $W/\langle e\rangle,$ it follows that
$C_{\tilde{G}}(W)\leq C_{P_{1}}(W/\langle e\rangle)$ . Hence 
\[
\chi(1)>\frac{|GU_{n}(q)|_{p'}}{|GU_{k}(q)|_{p'}|GU_{n-2-k}(q)|_{p'}|GL_{1}(q^{2})|_{p'}(q+1)}.
\]
But the right hand side of the above is greater than $\chi_{(n-3,2,1)}(1)+d$
when $k\geq1$. Hence we may assume that $W=U$ is totally isotropic.

Lastly, if $D$ preserves a non-degenerate subspace $U$ of $V$,
where $1\leq k=\dim U\leq\frac{n}{2}$, then we can assume $k\leq3:$
if $D$ preserves such a subspace then $D\leq GU_{k}\times GU_{n-k}$
and $\chi(1)$ is at least
\[
\frac{\prod_{i=1}^{n}(q^{i}-(-1)^{i})}{(q+1)\prod_{i=1}^{k}(q^{i}-(-1)^{i})\prod_{i=1}^{n-k}(q^{i}-(-1)^{i})}=\frac{\prod_{i=n-k+1}^{n}(q^{i}-(-1)^{i})}{(q+1)\prod_{i=1}^{k}(q^{i}-(-1)^{i})}.
\]
 Now if $k\geq4,$ 
\[
\frac{\prod_{i=n-k+1}^{n}(q^{i}-(-1)^{i})}{(q+1)\prod_{i=1}^{k}(q^{i}-(-1)^{i})}\geq q^{3n-7}
\]
 and so $\chi(1)>\chi_{1,(n-3,2,1)}+d$ as desired. 

It follows from the above assumptions that each $V_{i}$ is totally
isotropic of dimension 1 or non-degenerate of dimension $1,2,3,n-3,n-2$
or $n-1.$ After a number of very similar calculations when $n=7$
or $9$ we may also assume that $n_{l}\geq n-3$ and that $D$ preserves
no non-degenerate subspaces of $V_{l}$. Clearly then $l\leq4$ and
max$\{n_{i}\}_{i=1}^{l-1}\leq3$.

There now remain a number of subcases to consider. Firstly consider
the case where $l=2$, $n_{1}=3$ and $n_{2}=n-3$, and $D$ preserves
no non-degenerate subspaces of $V_{1}$.  If $\tilde{s}|_{V_{2}}$
is not a scalar it follows from the $l=1$ case that $D\leq GU_{3}(q)\times GU_{m}(q^{k})$
where $mk=n-3$ and $k\geq3$. Consequently
\begin{eqnarray*}
\chi(1) & > & \frac{(q^{n}+1)(q^{n-1}-1)(q^{n-2}+1)}{(q+1)(q^{2}-1)(q^{3}+1)}\cdot\frac{(q+1)(q^{2}-1)\dots(q^{mk}-1)}{(q+1)(q^{k}+1)\dots(q^{mk}-(-1)^{m})}\\
 & > & \frac{q^{3n-8}}{8(q+1)}\cdot q^{\frac{1}{2}m^{2}k(k-1)-1}>q^{3n-14+\frac{1}{2}m^{2}k(k-1)}
\end{eqnarray*}
by Lemma \ref{lem:Lemma on partition powers}. Evidently $\frac{m^{2}k(k-1)}{2}\geq m^{2}k=(n-3)k\geq2(n-3)\geq7$
and hence $\chi(1)>q^{3n-7}$. Hence we assume that $\tilde{s}$ does
act as a scalar on $V_{2}$. As $3\neq n-3$, we conclude that either
$C=D=GU_{3}(q)\times GU_{n-3}(q)$ if $\tilde{s}|_{V_{1}}$is a scalar,
or $C=D=GU_{1}(q^{3})\times GU_{n-3}(q)$ otherwise. Hence 
\[
\chi(1)\geq\psi(1)\cdot\frac{(q^{n}+1)(q^{n-1}-1)(q^{n-2}+1)}{(q^{3}+1)(q^{2}-1)(q+1)}
\]
 where $\psi(1)$ is the degree of the unipotent character of $C^{*}$
corresponding to $\chi$ as in Proposition \ref{prop:(Jordan-decomposition-of characters and degree formula}.
Now if $\psi|_{GU_{n-3}(q)}$ is non-trivial, it has degree at least
$\frac{q(q^{n-4}+1)}{(q+1)}$ by (\cite{Tiep and Zalesskii Minimal characters},
4.1) and $\chi(1)>\chi_{(n-3,2,1)}+d$, as required. To conclude this
case, we assume that $\psi|_{GU_{n-3}(q)}$ is trivial, and then it
is easy to find all possible degrees $\psi(1)$ as the unipotent character
degrees of $GU_{1}(q^{3})$ and $GU_{3}(q)$ are known (\cite{Carter rep theory},
13.8). With one exception we find that $d>|\chi(1)-\chi_{1,(n-3,2,1)}|$.
The exceptional case occurs when $(n,q)=(7,2)$, $C=GU_{3}\times GU_{4}(q)$
and $\psi$ corresponds to the multi-partition $((1,2),(4))$. Here
$\psi(1)=\frac{2\cdot(2^{2}-1)}{(2+1)}$ and $\chi(1)=\frac{2\cdot(2^{7}+1)(2^{6}-1)(2^{5}+1)}{(2^{3}+1)(2+1)^{2}}=6622$
as in the conclusion of the Lemma. 

Next consider the case where $\text{max}\{n_{i}\}_{i=1}^{l-1}=2$.
By the work above, it follows that $D\leq GU_{2}(q)\times GU_{n-2}(q)$
or $D\leq GU_{1}(q)\times GU_{2}(q)\times GU_{n-3}(q).$ Assuming
the former case for now, if $\tilde{s}|_{V_{2}}$ is non-scalar then
$D\leq GU_{2}(q)\times GU_{m}(q^{k})$ such that $m=\frac{n-2}{k}$
and $k\geq3$. Here 
\begin{eqnarray*}
\chi(1) & > & \frac{(q^{n}+1)(q^{n-1}-1)}{(q^{2}-1)(q+1)^{2}}\cdot\frac{(q+1)\dots(q^{n-2}+1)}{(q+1)(q^{k}+1)\dots(q^{n-2}-(-1)^{\frac{n-2}{k}})}\\
 & > & \frac{q^{2n-3}}{8(q+1)^{2}}\cdot q^{\frac{1}{2}m^{2}k(k-1)-1}>q^{3m-7}.
\end{eqnarray*}
 This leaves the case where $\tilde{s}|_{V_{2}}$ is a scalar. Assuming
additionally that $\tilde{s}|_{V_{1}}$ is also a scalar, then as
$n_{1}<n_{2}$ it follows that $C=D=GU_{2}(q)\times GU_{n-2}(q)$
and 
\[
\chi(1)=\psi(1)\cdot\frac{(q^{n}+1)(q^{n-1}-1)}{(q^{2}-1)(q+1)}.
\]
 Suppose that $\psi|_{GU_{n-2}(q)}$ is non-trivial. Then by (\cite{Tiep and Zalesskii Minimal characters},
4.1), $\psi(1)\geq\frac{q(q^{n-3}-1)}{(q+1)}$ and hence 
\[
\chi(1)\geq\frac{q(q^{n}+1)(q^{n-1}-1)(q^{n-3}-1)}{(q^{2}-1)(q+1)}.
\]
We can then check that this is greater than $\chi_{(n-3,2,1)}(1)+d$
as required, provided $(n,q)\neq(7,2)$. If however $(n,q)=(7,2)$
and $\psi(1)=\frac{2\cdot(2^{7-3}-1)}{(2+1)}$ then 
\[
|\chi(1)-\chi_{(4,2,1)}(1)|=|\frac{2\cdot(2^{7}+1)(2^{6}-1)(2^{4}+1)}{(2^{2}-1)(2+1)^{2}}-\chi_{(4,2,1)}(1)|\leq d
\]
 and we have another of the listed exceptions to the Lemma. Lastly,
if $\psi|_{GU_{n-2}(q)}=1$ then 
\[
\chi(1)\in\{\frac{q(q^{n}+1)(q^{n-1}-1)}{(q^{2}-1)(q+1)},\frac{(q^{n}+1)(q^{n-1}-1)}{(q^{2}-1)(q+1)}\}
\]
 and $\chi_{(n-3,2,1)}(1)-\chi(1)>d$. The remaining case where $\tilde{s}|_{V_{1}}$is
non-scalar and $D\leq GL_{1}(q^{2})\times GU_{n-2}(q)$ follows similarly
to the above. Indeed, if $\psi|_{GU_{n-2}(q)}$ is non-trivial then
\[
\chi(1)\geq\frac{q(q^{n}+1)(q^{n-1}-1)(q^{n-3}-1)}{(q^{2}-1)}>\chi_{(n-3,2,1)}(1)+d.
\]
 If however $\psi|_{GU_{n-2}(q)}=1$ then 
\[
\chi(1)=\frac{(q^{n}+1)(q^{n-1}-1)}{(q^{2}-1)(q+1)}<\chi_{(n-3,2,1)}(1)-d.
\]

To complete the proof we require a consideration of the cases where
$D\leq GU_{1}(q)\times GU_{2}(q)\times GU_{n-3}(q)$ or $D\leq GU_{1}(q)^{i}\times GU_{n-i}(q)$
for $i=1,2$ or $3$. We shall however omit the details as the proof
follows very similarly to the work above. In summary, we find that
apart for one exception, $|\chi_{(n-3,2,1)}(1)-\chi(1)|>d$. The single
exception occurs when $(n,q)=(7,2)$, $C=D=GU_{1}(q)^{2}\times GU_{5}(q)$
and $\psi$ corresponds to the multi-partition $((1),(1),(1,4))$.
In this case $\chi(1)=9030$.  
\end{proof}
\textbf{Proof of Proposition \ref{prop: The D_alpha is the unipotent we're looking for}: }
\begin{proof}
Recall that $\alpha=\chi_{q^{2}-q}^{(q+1)}\in$Irr$(GU_{3}(q))$.
Firstly note from (\ref{eq:max character degree difference}) in the
proof of Proposition \ref{prop:classsifying the level 3 characters}
that
\[
|D_{\alpha}^{\circ}(1)-\chi_{(n-3,2,1)}(1)|\leq(\omega|_{G},\zeta_{n,q})\zeta_{n,q}(1)/\alpha(1)=\frac{q^{n}(q^{3}+1)(q+1)}{q(q-1)}.
\]
It then follows from Lemmas \ref{prop: degree of unipotent characterss and proximity to D3}
and \ref{prop:Non-unipotent degrees close to my character} that either
$D_{\alpha}=D_{\alpha}^{\circ}=\chi_{(n-3,2,1)}$ as required, or
$(n,q)=(7,2)$ and 
\[
D_{\alpha}^{\circ}(1)\in\{6622,10\,234,9030\}.
\]
 Assume the latter statement for a contradiction and first note that
here $D_{\alpha}(1)=\chi_{(4,2,1)}(1)=7568.$ Therefore as $D_{\alpha}(1)\geq D_{\alpha}^{\circ}(1)$,
we can discard the second and third possibilities listed above as
they are too large and it follows that $D_{\alpha}^{\circ}(1)=6622$. 

From Table \ref{table:1} we see that $SU_{7}(2)$ has at least two
additional characters of degree 6622. Indeed, when $\beta=\chi_{q^{2}-q}^{(t)}$
where $t=1,2$, we have $D_{\beta}(1)=6622.$ Furthermore the irreducible
component $D_{\beta}^{\circ}\in\text{Irr}(SU_{7}(2))$ has degree
lying in the range $4894\leq D_{\beta}^{\circ}(1)\leq6622$ by Proposition
\ref{prop:classsifying the level 3 characters}. But the irreducible
character degrees of $SU_{7}(2)$ are known (\cite{lubeck su72 degrees})
and we check that there exist exactly two irreducible characters with
degrees in this range, a contradiction.
\end{proof}
\begin{lem}
\label{lem:x and y on level 3 unipotent character}Let $x$ and $y$
be the regular semisimple elements chosen in Proposition \ref{prop:picking x and y in n odd}.
Then $\chi_{(n-3,2,1)}(x)=\chi_{(n-3,2,1)}(y)=1$. Also, $|\chi_{(3,2,1^{n-5})}(x)|=|\chi_{(3,2,1^{n-5})}(y)|=1$.
\end{lem}
\begin{proof}
By Proposition \ref{prop: The D_alpha is the unipotent we're looking for},
$\chi_{(n-3,2,1)}=D_{\alpha}$ where $\alpha=\chi_{q^{2}-q}^{(q+1)}\in\text{Irr}(GU_{3}(q))$.
The values of $\alpha$ are known (\cite{Ennola}, Table 7), and recall
for $g\in G$ we have the formula
\[
D_{\alpha}(g)=\frac{1}{|S|}\sum_{h\in S}\overline{\alpha(h)}\omega(hg),
\]
 where $\omega=\zeta_{3n,q}$ is the reducible Weil character of $GU_{3n}(q)$.
Applying this formula to $x$ and $y$ then gives the values for $\chi_{(n-3,2,1)}$.
It follows from the formula above that $\chi_{(n-3,2,1)}$ is $\tilde{G}$-invariant
and hence extends to a character of $\tilde{G}$ by Lemma \ref{lem:cyclic extensions and extendible characters}.
The values of $|\chi_{(3,2,1^{n-5})}(x)|$ and $|\chi_{(3,2,1^{n-5})}(y)|$
then follow by Lemmas \ref{lem dual charracters fixed on reg sem si-1}
and \ref{lem tiep and bob result}, as in the proof of Proposition
\ref{prop:x and y cover the group}.
\end{proof}

\subsubsection{Calculating Structure Constants for $G=SU_{n}(q)$, $n\geq7$ odd}

We now have sufficient information about the characters in the set
$X\subset$Irr$(G)$ defined after Proposition \ref{prop:non vanishing characters}
to calculate structure constants for $G$. For $g\in G$, let $g=su$
denote the Jordan decomposition of $g$ with $s$ semisimple and $u$
unipotent. We use the shorthand notation $u\sim(n^{r_{n}},\dots,2^{r_{2}},1^{r_{1}})$
to indicate the Jordan form of $u$ and let $\sigma_{g}$ denote the
spectrum of eigenvalues of $g.$ The elements $x,y\in G$ will always
denote those chosen in Proposition \ref{prop:picking x and y in n odd}.
In this section we first prove the following result.
\begin{prop}
\label{prop:calculating structurre constants any g }Let $G=SU_{n}(q)$
where $n\geq7$ is odd and let $g=su\in G$. Assume that $u\neq1$
and also that $u\nsim(2,1^{n-2})$ when $s=1$. Then $g\in x^{G}y^{G}$.
\end{prop}
For ease of presentation we treat the cases where $g$ is unipotent
and non-unipotent separately.
\begin{lem}
\label{lem:calculating structure constants unipotent elements} Let
$1\neq u\in G$ be unipotent, and assume that $u\nsim(2,1^{n-2})$.
Then $u\in x^{G}y^{G}$.
\end{lem}
\begin{proof}
By Theorem \ref{thm:Burnsides formula} and Proposition \ref{prop:non vanishing characters},
$u\in x^{G}y^{G}$ if and only if
\[
\kappa(x^{G},y^{G},u^{G})=\sum_{\chi\in X\subset\text{Irr}(G)}\frac{\chi(x)\chi(y)\chi(u^{-1})}{\chi(1)}\neq0.
\]

Removing the contribution of the trivial character, it suffices to
show that
\begin{equation}
\sum_{\chi\in X\backslash\{1_{G}\}}\frac{|\chi(x)\chi(y)\chi(u^{-1})|}{\chi(1)}<1.\label{eq:Simpler formula for the structure constant}
\end{equation}
The values of $|\chi(x)\chi(y)|$ are known by Lemmas \ref{lem:Level 1 and 2 c haracters on x and y}
and \ref{lem:x and y on level 3 unipotent character}, and in all
cases $|\chi(x)\chi(y)|\leq1$. Furthermore, when $q=2$, $\chi(y)=0$
for all non-unipotent characters $\chi\in X\backslash\{1_{G}\}$ (Lemma
\ref{lem:Level 1 and 2 c haracters on x and y}). Hence for (\ref{eq:Simpler formula for the structure constant})
to hold it is sufficient to show that 
\[
\sum_{\chi\in X\backslash\{1_{G}\}\subset\text{Irr}(G)}\frac{|\chi(u)|}{\chi(1)}<1,\ \text{if }q\geq3,
\]
 and 
\[
\frac{|\chi_{(n-3,2,1)}(u)|}{\chi_{(n-3,2,1)}(1)}+\frac{|\chi_{(3,2,1^{n-5})}(u)|}{\chi_{(3,2,1^{n-5})}(1)}+\frac{|St(u)|}{St(1)}<1,\ \text{if }q=2.
\]
 In either case, denote the sum of character ratios on the left hand
side by $\Delta(u)$.

To prove that $\Delta(u)<1$, we bound the character ratio summands
on a case by case basis. Firstly consider the possible values of $\chi_{(n-3,2,1)}(u)$.
Let $u\sim(n^{r_{n}},\dots,2^{r_{2}},1^{r_{1}})$ and denote the total
number of blocks by $r$. Recall from Proposition \ref{prop: The D_alpha is the unipotent we're looking for}
that $\chi_{(n-3,2,1)}=D_{\alpha}$ where $\alpha=\chi_{q^{2}-q}^{(q+1)}\in\text{Irr}(GU_{3}(q))$
and hence 
\begin{equation}
D_{\alpha}(u)=\frac{1}{|GU_{3}(q)|}\sum_{h\in GU_{3}(q)}\overline{\alpha(h)}\omega(hu).\label{eq: values of the D3}
\end{equation}
 Here $\omega$ denotes the reducible Weil character of $GU_{3n}(q)$
that takes values
\[
\omega(hu)=-(-q)^{\text{dim Ker}(hu-1)}.
\]
 The values of $\alpha$ are known (\cite{Ennola}, Table 7) and we
compute
\begin{eqnarray}
|GU_{3}(q)|\cdot\chi_{(n-3,2,1)}(u) & = & (q^{2}-q)\left(-(-q)^{3r}-q\right)\nonumber \\
 & - & q\left(-(-q)^{3r-r_{1}}-q\right)(q-1)(q^{3}+1)\nonumber \\
 & - & q^{2}(q-1)(q^{2}-q+1)\left(-q^{2r+1}-(-q)^{r+1}-q(q-1)\right)\nonumber \\
 & + & q^{2}(q-1)(q^{3}+1)\left(-(-q)^{2r-r_{1}+1}-(-q)^{r+1}-q(q-1)\right)\nonumber \\
 & + & 2q^{3}(q-1)^{2}(q^{2}-q+1)\left(\frac{-3(-q)^{r+1}-q(q-2)}{6}\right)\nonumber \\
 & + & \frac{q^{4}}{3}(q+1)^{3}(q-1)^{2}.\label{eq:D3(U)-1}
\end{eqnarray}
 Recall that $\chi_{(n-3,2,1)}(1)\sim q^{3n-7}$. Hence when $r$
is large, $\frac{|\chi_{(n-3,2,1)}(u)|}{\chi_{(n-3,2,1)}(1)}\thicksim q^{3r-3n}$
and checking small cases computationally we see that $\frac{|\chi_{(n-3,2,1)}(u)|}{\chi_{(n-3,2,1)}(1)}<q^{-3}.$ 

We do not have a method for computing general character values of
$\chi_{(3,2,1^{n-5})}$ but it is enough to use the trivial bound
$|\chi_{(3,2,1^{n-5})}(g)|\leq|C_{G}(g)|^{\frac{1}{2}}$. Note that
$\chi_{(3,2,1^{n-5})}(1)\sim q^{\frac{1}{2}(n^{2}-n-8)}$ and unipotent
centralizers can be easily computed using the details in (\cite{Liebeck and Seitz book},
7.1). Hence we can bound the character ratio $\frac{|\chi_{(3,2,1^{n-5})}(u)|}{\chi_{(3,2,1^{n-5})}(1)}$.
These bounds will be useful for later computations so we list some
details in Table \ref{table:2}. \begin{table}[h] \begin{center} \caption{Character Ratio Bounds} \label{table:2}   \renewcommand{\arraystretch}{1.7} \begin{tabular}{ccc} & & \tabularnewline \hline  Block Structure & Conditions & Upper bound for  $\frac{|\chi_{(3,2,1^{n-5})}(u)|}{\chi_{(3,2,1^{n-5})}(1)}$ \tabularnewline \hline  $(2,1^{n-2})$ & $n\geq9,\,q=2,3$ & $\frac{1}{q^{2}}$\tabularnewline  & $n\geq9,\,q\geq4$ & $\frac{1}{q^{3}}$\tabularnewline \hline $(2^{2},1^{n-4})$ & $n=7,\,q=2$ & $0.263$\tabularnewline  & $n=7,\,q\geq3$ & $\frac{1}{q^{2}}$\tabularnewline  & $n\geq9$ & $\frac{1}{q^{4}}$\tabularnewline \hline $(3,1^{n-3})$ & $n=7$ & $\frac{1}{q^{2}}$\tabularnewline  & $n\geq9$ & $\frac{1}{q^{3}}$\tabularnewline \hline  \end{tabular} \end{center} \end{table}

For the unipotent elements not listed in Table \ref{table:2}, we
see that $|C_{G}(u)|^{\frac{1}{2}}\lesssim q^{\frac{1}{2}(n^{2}-6n+17)}$.
Checking small dimensions computationally we find in these cases that
$\frac{|\chi_{(3,2,1^{n-5})}(u)|}{\chi_{(3,2,1^{n-5})}(1)}<\frac{1}{q^{3}}$. 

Next we consider the contribution to $\Delta(u)$ of any character
\[
\chi\in\{\chi_{s^{t},(1^{n-1})},\chi_{s^{t},(2,1^{n-2})}\,|\,1\leq t\leq q\}\subset X.
\]
As stated above, these characters all vanish when $q=2$ but in general
we use the centralizer bound once more. In particular note that that
$\chi(1)>q\cdot\chi_{(3,2,1^{n-5})}(1)$ and hence the bounds found
above for $\frac{|\chi_{(3,2,1^{n-5})}(u)|}{\chi_{(3,2,1^{n-5})}(1)}$,
also apply to $q\cdot\text{\ensuremath{\frac{|\chi(u)|}{\chi(1)}} }$.
It follows that the sums $\sum_{t=1}^{q}\frac{|\chi_{s^{t},(1^{n-1})}(u)|}{\chi_{s^{t},(1^{n-1})}(1)}$
and $\sum_{t=1}^{q}\frac{|\chi_{s^{t},(2,1^{n-2})}(u)|}{\chi_{s^{t},(2,1^{n-2})}(1)}$
are bounded by the values given in Table \ref{table:2}, or $q^{-3}$
otherwise. 

The remaining characters to consider are the non-unipotent families
$\{\chi_{s^{t},(n-1)}\}_{t=1}^{q}$ and $\{\chi_{s^{t},(n-2,1)}\}_{t=1}^{q}$.
Let $\chi\in\{\chi_{s^{t},(n-1)}\}_{t=1}^{q}$. We can calculate character
values using Lemma \ref{lem:(REF-TIEP).-Formulae for Weil reps}.
We compute some of the largest values and print them below for elements
of interest. 
\[
u\sim(2,1^{n-2}):\,\,\,\sum_{t=1}^{q}\frac{|\chi_{s^{t},(n-1)}(u)|}{\chi_{s^{t},(n-1)}(1)}=\frac{q(q^{n-1}-1)}{q^{n}+1},
\]
\[
u\sim(2^{2},1^{n-4}),(3,1^{n-3}):\,\,\,\sum_{t=1}^{q}\frac{|\chi_{s^{t},(n-1)}(u)|}{\chi_{s^{t},(n-1)}(1)}=\frac{q(q^{n-2}+1)}{q^{n}+1}.
\]
 For all other unipotent elements it is an easy check that $\sum_{t=1}^{q}\frac{|\chi_{s^{t},(n-1)}(u)|}{\chi_{s^{t},(n-1)}(1)}<\frac{1}{q^{2}}$. 

Lastly let $\chi\in\{\chi_{s^{t},(n-2,1)}\}_{t=1}^{q}.$ Again we
have an explicit formula for character values, namely formula (\ref{eq:formula for D2})
in Section \ref{subsec:Values-of-characters}. We find for $u\sim(n^{r_{n}},\dots,2^{r_{2}},1^{r_{1}})$
with a total number of blocks $r$, that 
\begin{eqnarray*}
|GU_{2}(q)|\cdot\chi_{s^{t},(n-2,1)}(u) & = & (q-1)(q^{2r}-1)\\
 & - & (q^{2}-1)((-q)^{2r-r_{1}}-1)\\
 & + & q(q-1)((-q)^{r}(-q+1)+(q-1)).
\end{eqnarray*}
This is maximal when $u\sim(2,1^{n-2})$ and in this case 
\[
\sum_{t=1}^{q}\frac{|\chi_{s^{t},(n-2,1)}(u)|}{\chi_{s^{t},(n-2,1)}(1)}<q\cdot\frac{1}{q^{15/8}}=\frac{1}{q^{7/8}}.
\]
 For all other unipotent elements,
\[
\sum_{t=1}^{q}\frac{|\chi_{s^{t},(n-2,1)}(u)|}{\chi_{s^{t},(n-2,1)}(1)}<\frac{1}{q^{5/2}}.
\]

Collating these bounds for character ratios, we can compute upper
bounds for $\Delta(u).$ Firstly let $q=2$. Here 
\[
\Delta(u)=\frac{|\chi_{(n-3,2,1)}(u)|}{\chi_{(n-3,2,1)}(1)}+\frac{|\chi_{(3,2,1^{n-5})}(u)|}{\chi_{(3,2,1^{n-5})}(1)}+\frac{|St(u)|}{St(1)}.
\]
Using the bounds above and noting that $St(u)=0$ (\cite{Carter rep theory},
6.4.7) gives
\[
\Delta(u)<0.263+\frac{1}{2^{3}}<1.
\]
Similarly, when $q\geq3$ the bounds above give 
\[
\Delta(u)=\sum_{\chi\in X\backslash\{1_{G}\}\subset\text{Irr}(G)}\frac{|\chi(u)|}{\chi(1)}<\frac{q(q^{n-2}+1)}{(q^{n}+1)}+\frac{3}{q^{2}}+\frac{1}{q^{5/2}}+\frac{1}{q^{3}}<1.
\]
 
\end{proof}
\begin{lem}
\label{lem:structure constant non-unipotent} Let $g=su\in G$ such
that $s,u\neq1$. Then $g\in x^{G}y^{G}$.
\end{lem}
\begin{proof}
As explained in the proof of Lemma \ref{lem:calculating structure constants unipotent elements},
for the result to hold it suffices to show that $\Delta(g)<1$ where
\[
\Delta(g)=\sum_{\chi\in X\backslash\{1_{G}\}}\frac{|\chi(g)|}{\chi(1)},\ \text{if }q\geq3,
\]
 and
\[
\Delta(g)=\frac{|\chi_{(n-3,2,1)}(g)|}{\chi_{(n-3,2,1)}(1)}+\frac{|\chi_{(3,2,1^{n-5})}(g)|}{\chi_{(3,2,1^{n-5})}(1)}+\frac{|St(g)|}{St(1)},\ \text{if }q=2.
\]

Firstly consider $\chi\in\{\chi_{(3,2,1^{n-5})}\}\cup\{\chi_{s^{t},(1^{n-1})},\chi_{s^{t},(2,1^{n-2)}}\,|\,1\leq t\leq q\}\subset X$.
As $|C_{G}(g)|\leq|C_{G}(u)|$ we can use the bounds for $\frac{|\chi_{(3,2,1^{n-5})}(g)|}{\chi(1)}$
calculated in Lemma \ref{lem:calculating structure constants unipotent elements}.
Note however that Table \ref{table:2} does not include a treatment
of the case $n=7$, $u\sim(2,1^{5})$. However when $s\neq1$ and
$u\sim(2,1^{5})$ we can sufficiently bound the character ratio by
estimating $|C_{G}(s)|^{1/2}$ and find that $\frac{|\chi_{(3,2,1^{n-5})}(g)|}{\chi(1)}<\frac{1}{q^{2}}$
in this case. 

The value of $|\chi_{(n-3,2,1)}(g)|$ has a slightly more complicated
formula when $s\neq1$ but we can again find a sufficient bound. Let
$\sigma_{g}=\{\lambda_{i}\}$ denote the eigenvalues of $g$ and note
these are not necessarily in $\mathbb{F}_{q^{2}}.$ Let $\delta,\rho,\tau\in\mathbb{F}_{q^{6}}$
be elements of orders $q+1,q^{2}-1$ and $q^{3}+1$ respectively.
Then define $\sigma_{g,\delta}:=\sigma_{g}\cap\langle\delta\rangle$,
$\sigma_{g,\rho}:=\{\sigma_{g}\cap\langle\rho\rangle\}\backslash\sigma_{g,\delta}$
and $\sigma_{g,\tau}:=\{\sigma_{g}\cap\langle\tau\rangle\}\backslash\sigma_{g,\rho}.$
From the presence of the Weil character $\omega$ in formula (\ref{eq: values of the D3}),
we see that the magnitude of $|\chi_{(n-3,2,1)}(g)|$ is controlled
by the dimension of the eigenspaces of each $\lambda_{i}\in\sigma_{g,\delta},\sigma_{g,\rho},\sigma_{g,\tau}$. 

Let $m_{\delta}:=\text{max}_{\lambda\in\sigma_{g,\delta}}\{\text{dim(Ker}(g-\lambda I))\}$
and define $m_{\rho}$ and $m_{\tau}$ similarly. Then 
\[
|\chi_{(n-3,2,1)}(g)|\sim q^{\text{max}_{i}\{3m_{\delta}+2,2m_{\rho}+5,3m_{\tau}+6\}-7}.
\]
In general this is maximal when $g$ has Jordan normal form $(\lambda J_{1}^{n-2})\oplus(\mu J_{2})$
for $\lambda,\mu\in\sigma_{g,\delta}$ such that $\lambda\neq\mu$.
Checking small values of $(n,q)$ explicitly we find that $\frac{|\chi_{(n-3,2,1)}(g)|}{\chi_{(n-3,2,1)}(1)}\leq\frac{1}{q^{3}}$.

Finally consider $\chi\in\{\chi_{s^{t},(n-1)}\}_{t=1}^{q}\cup\{\chi_{s^{t},(n-2,1)}\}_{t=1}^{q}.$
Here the character values are given by Lemma \ref{lem:(REF-TIEP).-Formulae for Weil reps}
and formula (\ref{eq:formula for D2}) in Section \ref{subsec:Values-of-characters}.
Much like the above, $|\chi(g)|$ has a non-trivial contribution for
each $\lambda_{i}\in\sigma_{g,\delta},\sigma_{g,\rho}$ and is maximised
when $g$ has Jordan normal form $(\lambda J_{1}^{n-2})\oplus(\mu J_{2})$
for $\lambda,\mu\in\sigma_{g,\delta}$ such that $\lambda\neq\mu$.
Checking small cases computationally we find that 

\[
\sum_{t=1}^{q}\frac{|\chi_{s^{t},(n-1)}(g)|}{\chi_{s^{t},(n-1)}(1)}=\frac{q^{n-2}+q+2}{q^{n}+1}<q^{-15/8}
\]
 and 
\[
\sum_{t=1}^{q}\frac{|\chi_{s^{t},(n-2,1)}(u)|}{\chi_{s^{t},(n-2,1)}(1)}<\frac{1}{q^{2}}.
\]

We collate the information above to bound $\Delta(g)$. Firstly let
$q=2$. Again, as $g$ is not semisimple, $St(g)=0$ and it follows
that 
\[
\Delta(g)<0.263+\frac{1}{2^{3}}<1.
\]
Now let $q\geq3$. By the above it follows that 
\[
\Delta(g)<\frac{1}{q^{15/8}}+\frac{3}{q^{2}}+\frac{1}{q^{2}}+\frac{1}{q^{3}}<1.
\]
\end{proof}
Lemmas \ref{lem:calculating structure constants unipotent elements}
and \ref{lem:structure constant non-unipotent} together complete
the proof of Proposition \ref{prop:calculating structurre constants any g }.

We can now finally prove Theorem \ref{thm:inv wwidth of PSU} where
$n$ is odd, to conclude our work on simple unitary groups.\\
\\
\textbf{Proof of Theorem }\ref{thm:inv wwidth of PSU} ($n$ odd)\textbf{:}
\begin{proof}
We have already seen that the result holds when $n=3,5$ (Lemmas \ref{lem:SU result for 3}
and \ref{lem:psu5}) so we can assume that $n\geq7$.

Let $g\in SU_{n}(q)$ with Jordan decomposition $g=su$. If $u=1$
then it follows from Theorem \ref{thm:(Gow-)Theorem for ss elements}
that $\overline{g}=\overline{s}\in\overline{x}^{\overline{G}}\overline{y}^{\overline{G}}$.
Similarly, provided $u\nsim(2,1^{n-2})$ when $s=1$, $\overline{g}\in\overline{x}^{\overline{G}}\overline{y}^{\overline{G}}$
by Lemma \ref{prop:calculating structurre constants any g }. In either
case it then follows that iw$(\overline{g})\leq4$ by Proposition
\ref{prop:picking x and y in n odd}. This leaves the unipotent element
$g=u\sim(2,1^{n-2})$. Here we can embed $u$ in a subgroup $SU_{3}(q)$
when $q>2$, and $SU_{4}(q)$ when $q=2$. It follows that iw$(u)\leq4$
by Lemma \ref{lem:SU result for 3} and our earlier work in Section
\ref{subsec:n-even} on even-dimensional unitary groups. This completes
the proof of Theorem \ref{thm:inv wwidth of PSU}.
\end{proof}

\section{Exceptional Grou\label{sec:Simple-groups-of exceptional type}ps
of Lie type}

In this final section we consider the involution width of the exceptional
groups of Lie type. This will complete our case by case study via
the classification of finite simple groups and finish the proof of
Theorem~\ref{mainthm}. 

\subsubsection*{$G=E_{8}(q)$}

The first case $E_{8}(q)$ is illustrative of the method used for
the majority of the exceptional groups. Firstly choose regular semisimple
elements $x,\,y\in G$ of orders $r=$ppd$(q,\,24)$ and $s=$ppd$(q,\,20)$
respectively. Then $G\backslash\{1\}\subseteq x^{G}y^{G}$ by Theorem
7.6 of \cite{Guralnick and Malle}.

Classes of maximal tori in $G$ correspond to classes in the associated
Weyl group $W=W(E_{8})$ and orders of such tori are given by Carter
\cite{Carter Conjugacy classes in the weyl group}. By considering
possible orders, we see that $x$ lies in a unique maximal torus $T_{w}$
such that $|T_{w}|=\Phi_{24}(q)=q^{8}-q^{4}+1$. In this instance
$T_{w}$ corresponds to an element $w\in W$ of order $24.$

We claim that $x$ is strongly real. For this it suffices to show
there exists an involution in $G$ inverting all elements of $T_{w}$.
Now $N_{G}(T_{w})/T_{w}\cong C_{W}(w)$ and thus contains the coset
corresponding to the longest element of the Weyl group, namely $w_{0}=-1$.
This central element acts by inversion on the torus $T_{w}$ and hence
so does any preimage $n_{0}\in N_{G}(T_{w})$. As $|T_{w}|$ is odd
and $w_{0}$ is an involution, there exists a preimage $n_{0}$ that
is also an involution. 

The argument for the element $y$ is almost identical. Here $y$ is
contained in a torus $T_{w}$ of order $\Phi_{20}(q)=q^{8}-q^{6}+q^{4}-q^{2}+1$
which is also odd. Hence as before $y$ is strongly real, and so iw$(G)\leq4$
as required. 

\subsubsection*{$G=$$\,{}^{2}B_{2}(2^{2n+1}),$ $n\geq1$}

Next consider the Suzuki groups $^{2}B_{2}(2^{2n+1})$. We take elements
$x$ and $y$ both of order $r=$ppd$(2^{2n+1},\,4)$ and it follows
by Theorem 7.1 of \cite{Guralnick and Malle} that $G\backslash\{1\}\subseteq x^{G}y^{G}$.
Suzuki \cite{Suzuki construction} showed that $G$ has $q+3$ conjugacy
classes of which only two, containing order 4 elements, are not strongly
real. Thus $x$ and $y$ are strongly real, and so iw$(G)\leq4$. 

\subsubsection*{$G=G_{2}(q),\,q\geq3$}

For $q\neq4$ we take elements $x$ and $y$ of order $r=$ ppd$(q,\,3)$
and it follows by Theorem 7.3 of \cite{Guralnick and Malle} that
$G\backslash\{1\}\subseteq x^{G}y^{G}.$ As $x$ and $y$ are regular
semisimple they lie in a unique maximal torus $T$, which by choice
of $r$ has size $q^{2}+q+1$. Now $W(G_{2})\cong D_{12}$ has longest
element $w_{0}=-1$ and thus as $T$ has odd order , we follow the
argument given in the case of $E_{8}(q)$ to see that $x$ and $y$
are strongly real. Finally, when $q=4$ we check using GAP \cite{GAP}
that $G_{2}(4)$ has involution width $3$. 

\subsubsection*{$G=$$\,{}^{2}G(q)$ }

Consider $G=$$\,{}^{2}G_{2}(q)$ with $q=3^{2n+1}>3.$ Here we take
$x$ of order $r=$ ppd$(q,\,6)$ and it follows by (\cite{Guralnick and Malle},
Thm 7.1) that $G\backslash\{1\}\subseteq x^{G}x^{G}$. The full character
table is known due to Ward \cite{Ward on Ree} and checking orders,
$x$ belongs to one of two classes, namely classes $V$ or $W$ (see
\cite{Ward on Ree} for notation). These classes contain elements
of orders $q+\sqrt{3q}+1$ and $q-\sqrt{3q}+1$ respectively. It is
now straightforward to show that $x$ is strongly real by computing
structure constants. Letting $a$ be a representative of the single
conjugacy class of involutions (this has size $q(q-1)(q+1)$), there
exist exactly four characters $\chi\in$Irr$(G)$ such that $\chi(a)\chi(x)\neq0$.
By Theorem \ref{thm:Burnsides formula} 
\begin{eqnarray*}
\kappa(a^{G},a^{G},x^{G}) & = & \sum_{\chi}\frac{\chi(a)^{2}\overline{\chi(x)}}{\chi(1)}\\
 & = & 1-\frac{1}{3^{2n+1}}+\frac{3^{2n+1}-1}{3^{n}(3^{2n+1}+1\pm3^{n+1})}>0.
\end{eqnarray*}

Hence $x\in(a^{G})^{2}$ and it follows that iw$(G)\leq4$.

\subsubsection*{$G=$$\,{}^{3}D_{4}(q)$ }

This group is strongly real (see Theorem \ref{Thm: Strongly real G}). 

\subsubsection*{$G=$$\,{}^{2}F_{4}(2^{2k+1})',$ $k\geq0$}

Let $G=$$\,{}^{2}F_{4}(q)'$ with $q=2^{2k+1}$. If $q=2$ then iw$(G)=3$
by GAP \cite{GAP}, so assume $q>2.$ Lemma $2.13$ of \cite{Guralnick and Tiep}
shows that $G$ contains regular semisimple elements $x$ of order
$r=$ ppd$(2^{2k+1},\,12)$ and $y$ of order $s=$ ppd$(2^{2k+1},\,6)$,
such that $x^{G}y^{G}=G\backslash\{1\}$. In the notation of Shinoda
\cite{Shinoda}, $x$ is conjugate to $t_{16}$ or $t_{17}$ and $y$
is conjugate to $t_{15}$. These lie in unique maximal tori, namely
$T(10)\cong\mathbb{Z}_{q^{4}-\sqrt{2}q^{3}+q^{2}-\sqrt{2}q+1}$ or
$T(11)\cong\mathbb{Z}_{q^{4}+\sqrt{2}q^{3}+q^{2}+\sqrt{2}q+1}$ and
$T(9)\cong\mathbb{Z}_{q^{4}-q^{2}+1}$ respectively. 

Note that each of these tori is of odd order. Therefore as the Weyl
group $W(G)\cong D_{16}$ again contains the central involution $w_{0}=-1$
we can follow the argument for $E_{8}(q)$ to show that $x$ and $y$
are strongly real.

\subsubsection*{$G=F_{4}(q)$ }

Take regular semisimple elements $x$ and $y$ of orders $r=$ ppd$(q,\,12)$
and $s=$ ppd$(q,\,8)$ respectively and it follows that $G\backslash\{1\}\subseteq x^{G}y^{G}$
by (\cite{Guralnick and Malle}, 7.6). Note that these are elements
of maximal tori of order $|T_{1}|=\Phi_{12}(q)$ and $|T_{2}|=\Phi_{8}(q).$

If $q$ is even then both tori have odd order and therefore a preimage
$n_{0}\in G$ of $w_{0}=-1$ can also be taken to be an involution.
If $q$ is odd then (\cite{Singh and Thakur}, Thm.$2.3.3$) shows
that all semisimple element of $F_{4}(q)$ are strongly real and hence
the result follows.

\subsubsection*{$G=E_{7}(q)$}

Let $G=E_{7}(q)$ and let $r=(\Phi_{2}(q)\Phi_{18}(q))_{\{2,3\}'}$
and $s=\Phi_{7}(q)$. By (\cite{Guralnick and Malle}, 7.7) there
exists regular semisimple $x$ and $y\in G$ of orders $r$ and $s$
respectively, such that $G\backslash\{1\}\subseteq x^{G}\cdot y^{G}$.
Let $G_{ad}=E_{7}(q).(2,q-1)$ denote the adjoint group of type $E_{7}$,
that is, the simple group with additional diagonal automorphism of
order $(2,q-1)$. By (\cite{Singh and Thakur}, Thm 2.3.3), every
semisimple element of $G_{ad}$ is strongly real in $G_{ad}$. Therefore,
when $q$ is even and so $G=G_{ad}$, it follows that $x$ and $y$
are strongly real and iw$(G)\leq4$. Now assume that the characteristic
is odd. The structures of maximal tori of $G_{ad}$ are given in \cite{Derizotis E type exceptional groups}
and we check that $x\in T_{1,ad}$ where $T_{1,ad}$ is cyclic of
order $m=(q+1)(q^{6}-q^{3}+1).$ Again by (\cite{Singh and Thakur},
Thm 2.3.3), there exists an involution $n_{0}\in G_{ad}$ such that
$x^{n}=x^{-1}$. In particular, if $\langle t\rangle=T_{1,ad}$ then
$x$ is contained in the dihedral subgroup $D_{2m}=\langle t,n_{0}\rangle$.
Note that $t^{k}n_{0}$ is an involution for all $k$ and inverts
$x$ by conjugation. Thus if $D_{2m}\cap E_{7}(q)$ contains such
an element then $x$ is strongly real in the simple group. Assuming
otherwise, it follows that $D_{2m}\cap E_{7}(q)\subset T_{1,ad}\cap E_{7}(q)$
and hence $|D_{2m}\cap E_{7}(q)|\leq\frac{m}{2}$ . This is a contradiction
and hence $x$ is indeed strongly real in $E_{7}(q)$. The same conclusion
for $y$ follows identically as the maximal torus $T_{2,ad},$ containing
$y$ and of order $q^{7}-1$, is again cyclic.

\subsubsection*{$G=E_{6}^{\epsilon}(q)$}

The final families of exceptional groups $E_{6}(q)$ and $^{2}E_{6}(q)$
require a more careful consideration. This is because there no longer
exists a central element $-1$ in the Weyl group $W(E_{6})$ and we
thus have to look harder for strongly real elements. Instead, for
root systems of type $E_{6}$ (as well as $A_{l}$, and $D_{l}$ for
$l$ odd) the longest element $w_{0}$ corresponds to the product
of $-1$ and the nontrivial symmetry of the Dynkin diagram.

Let $\boldsymbol{G}=E_{6}(\overline{\mathbb{F}}_{q})_{ad}$ denote
the adjoint algebraic group of type $E_{6}$ and $F:\boldsymbol{G}\rightarrow\boldsymbol{G}$
a Frobenius endomorphism such that $G=\boldsymbol{G}^{F}$ is the
finite adjoint group $E_{6}(q)_{ad}$ or $^{2}E_{6}(q)_{ad}$. These
finite groups are not necessarily simple but the derived groups $G'=E_{6}(q)$
or $^{2}E_{6}(q)$ are simple. We will use the notation $E_{6}^{\epsilon}(q)$
where $\epsilon\in\{+,-\}$, to denote $E_{6}(q)$ if $\epsilon=+$
and $^{2}E_{6}(q)$ if $\epsilon=-$. Define $d:=(3,q-\epsilon)$
and note that $|G\,:\,G^{'}|=d$.

In the remainder of this section we prove the following result.
\begin{thm}
\label{thm:E6 involution width result}Every element in $G'=E_{6}^{\epsilon}(q)$
can be written as a product of at most 4 involutions.
\end{thm}
As before, the strategy is to pick two strongly real classes $x^{G'}$
and $y^{G'}$ such that the product $x^{G'}\cdot y^{G'}$ covers as
much of the group $G'=E_{6}^{\epsilon}(q)$ as possible. If however
there exists $g\in G'$ for which we cannot show that $g\in x^{G'}\cdot y^{G'},$
then we embed $g$ in a subgroup $X\subset G'$, where $X$ is a Lie
type group for which the involution width is already known. Throughout
this section we assume that $q\neq2$. The Theorem can be checked
for $E_{6}^{\epsilon}(2)$ using GAP \cite{GAP}.\\

Choose $x,y\in G'$ with $x$ of order $r=$ppd$(q,\,12)$ and $y$
of order $s=$ppd$(q,\,8)$ . Note that $r$ and $s$ divide $\Phi_{12}(q)=q^{4}-q^{2}+1$
and $\Phi_{8}(q)=q^{4}+1$ respectively. The element $x$ is contained
in a Coxeter torus of a subgroup $F_{4}(q)$, and $|C_{G'}(x)|=(q^{4}-q^{2}+1)(q^{2}+q+1)/d$,
while $y$ is a regular semisimple element contained in the unique
maximal torus $C_{G'}(y)=T$ where $|T|=(q^{4}+1)(q^{2}-1)/d$.
\begin{lem}
\label{lem:x and y chosen in E6 are strongly real}The elements $x$
and $y$ are strongly real in $G'=E_{6}^{\epsilon}(q).$
\end{lem}
\begin{proof}
Firstly $x$ is contained in a maximal torus of the subgroup $F_{4}(q)$.
This torus has odd order $\Phi_{12}(q)$ and hence $x$ is inverted
by an involution contained in $F_{4}(q)$. Therefore $x$ is also
strongly real in the full group $E_{6}^{\epsilon}(q)$. 

We can also embed $y$ in a subgroup, namely a spin group $D_{5}^{\epsilon}(q)\subset E_{6}^{\epsilon}(q)$.
Note that in $D_{5}^{\epsilon}(q)$, $y$ is contained in a maximal
torus of order $(q^{4}+1)(q+\epsilon).$ Let $\overline{y}$ denote
the image of $y$ in $\Omega_{10}^{\epsilon}(q)$. It then follows
that $\overline{y}$ is strongly real in $\Omega_{10}^{\epsilon}(q)$
by (\cite{Malle and Saxl and Weigel}, 2.5(c) and 2.6 (c)) and we
let $\overline{t}\in\Omega_{10}^{\epsilon}(q)$ be an involution inverting
$\overline{y}.$ The involutions of $\Omega_{10}^{\epsilon}(q)$ that
lift to involutions in $D_{5}^{\epsilon}(q)$ are those where the
dimension of the negative eigenspace is divisible by 4 (\cite{Borovik},
8.4). This is true for $\overline{t}$ in our case as in $SO_{10}(\overline{\mathbb{F}}_{q})$,
$y$ is conjugate to an element of the form diag$(\lambda,\lambda^{q},\lambda^{q^{2}},\lambda^{q^{3}},\lambda^{-1},\lambda^{-q},\lambda^{-q^{2}},\lambda^{-q^{3}},1,1)$
where $|\lambda|=s$. Hence $\overline{t}$ lifts to an involution
$t\in D_{5}^{\epsilon}(q)$ that inverts $y$ by conjugation. 
\end{proof}
As usual, to show that a given conjugacy class $g^{G'}$ is contained
in the product $x^{G'}y^{G'}$, we compute the structure constants
using Theorem \ref{thm:Burnsides formula}. As in the proof of Theorem
\ref{thm:inv wwidth of PSU}, the first step is to reduce to the case
of unipotent characters. 

Firstly consider the untwisted case $G'=E_{6}(q)$. In the following
lemma, $St$ denotes the Steinberg character, and there are a further
two characters of interest, namely $D_{4,1}$ and $D_{4,\epsilon}$.
These characters arise from cuspidal unipotent characters of the Levi
subgroup $D_{4}(q)$ of $G'$. Full details are available in (\cite{Carter rep theory},
Sec. 13.9).
\begin{lem}
\label{E6 reduction to four characters in the structure constant sum}Suppose
$\chi\in Irr(G')$ such that $\chi(x)\chi(y)\ne0$. Then $\chi$ is
unipotent and $\chi\in\{1,\,St,\,D_{4,1},\,D_{4,\epsilon}\}$.
\end{lem}
\begin{proof}
Recall that the irreducible characters are partitioned into Lusztig
series as described in Section \ref{subsec:Preliminary-Material for general Lie Type Groups}.
By Lemma \ref{lem:G and M lemma on the existence of a character in a non trivial Lusztig series}
there exists a semisimple class $(t)$ such that $\chi\in\mathcal{{E}}(G',\,(t))$
and we firstly assume that $t=1$. Here, $\chi$ is by definition
a unipotent character and the degrees of such characters are given
in (\cite{Carter rep theory}, 13.9). It is easy to check that, excluding
the characters $1,\,St,\,D_{4,1}$ and $D_{4,\epsilon}$, all unipotent
characters are of defect zero for either $r$ or $s$. It therefore
follows from Theorem \ref{thm:Brauers theorem on prime defect} that
$\chi$ vanishes on either $x$ or $y$ for $\chi\notin\{1,\,St,\,D_{4,1},\,D_{4,\epsilon}\}$.
Now we consider $t\neq1$. In this case we use the formula given in
Proposition \ref{prop:(Jordan-decomposition-of characters and degree formula}
to find $\chi(1)$ and in particular we note that for a given prime
$l$, $\chi$ will have $l$-defect zero if $|C_{G'^{*}}(t)|_{l}=1$.
But the orders of centralizers of semisimple elements are known (\cite{Deriziotis and liebeck twisted e6},
Table 1 and \cite{Derizoitis untwisted e6} ,Table 4) and we check
that for all $t$, $|C_{G'^{*}}(t)|_{l}=1$ for either $l=r$ or $l=s$.
Hence $\chi$ vanishes on either $x$ or $y$ by Theorem \ref{thm:Brauers theorem on prime defect}.
\end{proof}
\begin{lem}
Suppose that $g\in G'$ such that $|C_{G'}(g)|\leq q^{26}$. Then
$g\in x^{G}\cdot y^{G}$.\label{prop:dim less that 26 then result}
\end{lem}
\begin{proof}
By Lemma \ref{E6 reduction to four characters in the structure constant sum}
we can so far evaluate the normalised structure constant as follows
\[
\kappa(x^{G'},y^{G'},g^{G'})=1+\sum_{\chi\in X}\frac{\chi(x)\chi(y)\overline{\chi(g)}}{\chi(1)},
\]
 where $X=\{D_{4,1},D_{4,\epsilon},St\}.$\\
Furthermore these characters have degrees
\[
St(1)=q^{36},\ D_{4,1}(1)=\frac{q^{3}}{2}\Phi_{1}^{4}\Phi_{3}^{2}\Phi_{5}\Phi_{9},\ D_{4,\epsilon}(1)=\frac{q^{15}}{2}\Phi_{1}^{4}\Phi_{3}^{2}\Phi_{5}\Phi_{9},
\]
 where $\Phi_{i}=\Phi_{i}(q)$ is the cyclotomic polynomial. 

Recall that $St(x)$,$St(y)\in\{\pm1\}$ (\cite{Carter rep theory},
6.4.7), and also note the trivial bound $|\chi(g)|\leq|C_{G'}(g)|^{\frac{1}{2}}\leq q^{13}$
for all $\chi\in$Irr$(G')$. Thus as 
\[
|C_{G'}(x)|\cdot|C_{G'}(y)|\leq(q^{4}+1)(q^{2}-1)(q^{4}-q^{2}+1)(q^{2}+q+1)\leq q^{13},
\]
 it follows that 
\[
|\sum_{\chi\in X}\frac{\chi(x)\chi(y)\overline{\chi(g)}}{\chi(1)}|\leq\frac{q^{\frac{13}{2}}|D_{4,1}(g)|}{\frac{1}{2}q^{3}\Phi_{1}^{4}\Phi_{3}^{2}\Phi_{5}\Phi_{9}}+\frac{q^{\frac{13}{2}}|D_{4,\epsilon}(g)|}{\frac{1}{2}q^{15}\Phi_{1}^{4}\Phi_{3}^{2}\Phi_{5}\Phi_{9}}+\frac{|St(g)|}{q^{36}}.
\]
 The right hand side of the above is then strictly less than $1$
as $|D_{4,1}(g)|$, $|D_{4,\epsilon}(g)|$ and $|St(g)|$ are all
at most $q^{13}$ by the hypothisis $|C_{G'}(g)|\leq q^{26}$. Hence
$\kappa(x^{G},y^{G},g^{G})\neq0$ and $g^{G^{'}}\in x^{G^{'}}y^{G^{'}}$.
\end{proof}
We can conclude by Lemma \ref{lem:x and y chosen in E6 are strongly real}
that if $|C_{G'}(g)|\leq q^{26}$ then $g$ has involution width at
most 4. The result also follows if $g=s$ is semisimple as here $g\in x^{G^{'}}y^{G^{'}}$follows
from Theorem \ref{thm:(Gow-)Theorem for ss elements}.To prove the
result for the remaining $g\in E_{6}(q)$, we embed the centralizer
(and therefore $g$) in a group of Lie type for which the involution
width is already known.

Recall that $\boldsymbol{G}$ denotes the adjoint algebraic group
$E_{6}(\bar{\mathbb{F}}_{q})_{ad}$ and (as we are currently assuming
$G'=E_{6}(q)$), $F$ is the standard Frobenius map.
\begin{lem}
\label{prop:e6 width proof when the centraliser is too big}Let $g\in E_{6}(q)$
such that $|C_{G'}(g)|>q^{26}$. Then $g$ has involution width at
most 4.
\end{lem}
\begin{proof}
Let $g=su$ be the usual Jordan decomposition. By the above we can
assume that $u\neq1$. Furthermore, as $C_{G'}(su)=C_{G'}(s)\cap C_{G'}(u)$,
it follows that $|C_{G'}(s)|,|C_{G'}(u)|>q^{26}$.

\textbf{Case 1: }\textit{Unipotent elements}: Assume $g=u$. The centralizer
sizes of unipotent elements are given by Mizuno (\cite{Mizuno on E6},
Section 4). For a given centralizer, representatives of the associated
class are given in terms of root elements of $G$. In particular the
$\boldsymbol{G}$-class of $g=u$ has one of the following labels:
\[
A_{1},A_{1}^{2},A_{2},A_{1}^{3},A_{2}A_{1},A_{2}A_{1}^{2},A_{2}^{2}.
\]
 In each case, $u$ is a distinguished unipotent element in a Levi
subgroup $\boldsymbol{L}$ of $\boldsymbol{G}$ corresponding to the
label. Denoting $C=C_{\boldsymbol{G}}(u)$, the number of $G$-classes
in $u^{\boldsymbol{G}}\cap G$ is equal to the number of classes in
$C/C^{0}$ by Lang's Theorem (\cite{Liebeck and Seitz book}, 2.12).
In all cases except when $u$ is of type $A_{2}$, $C=C^{0}$ and
hence $u^{\boldsymbol{G}}\cap G$ is a single class in $G$. Furthermore,
in such cases $\boldsymbol{L}$ is $F$-stable and $u\in\boldsymbol{L}'^{F}$.
All such subgroups lie in a subsystem subgroup $A_{2}(q)^{3}$ of
$G$ and so the result follows from Theorem \ref{thm:sl 4 ref}. Note
in the case where $u$ has label $A_{2}^{2}$, $u^{\boldsymbol{G}}\cap G'$
splits into $d=(3,q-1)$ classes. However because $N_{G}(A_{2}(q)^{3})$
contains the diagonal automorphism we may still embed all conjugacy
class representatives of $u^{\boldsymbol{G}}\cap G'$ in a subgroup
$A_{2}(q)^{3}$. This is not of concern otherwise, as it is only in
this case mentioned that the class splits in the simple group. Next
assume that $\boldsymbol{L}$ has label $A_{2}$. Here $C/C^{0}=\mathbb{Z}_{2}$
and thus $u^{\boldsymbol{G}}\cap G$ splits into 2 $G$-classes. Let
us adopt the usual notation where $\alpha=c_{1}\dots c_{6}$ denotes
the root $\Sigma c_{i}\alpha_{i}$ (here $\alpha_{i}$ refer to the
fundamental roots of the $E_{6}$ system) and $x_{\alpha}(t)$ is
a corresponding root element. Then by \cite{Mizuno on E6}, there
exist representatives of the two classes of the form
\[
x_{2}=x_{100000}(1)x_{001000}(1),
\]
\[
x_{21}=x_{100000}(1)x_{000100}(1)x_{000001}(1)x_{122321}(\zeta),\ p\neq2,
\]
\[
x_{40}=x_{100000}(1)x_{001000}(1)x_{000010}(1)x_{010110}(\eta),\ p=2,
\]
where $\zeta$ is a fixed non square and $x^{2}-x+\eta$ is irreducible
over $\mathbb{F}_{q}$. Now $x_{2}$ lies in a Levi subgroup $(\boldsymbol{L}')^{F}=A_{2}(q)$
and the result follows as before. For $x_{21}$, $\alpha=122321$
denotes the longest root of the $E_{6}$ system and hence the four
roots span an $A_{1}^{4}$ subsystem. It follows that $x_{21}$ is
contained in $A_{1}(q)^{4}\subset D_{4}(q)$, a spin group. The unipotent
element $x_{21}$ is uniquely determined by its Jordan decomposition
which has form $(J_{3}^{2},J_{1}^{2})$ on the natural $8$-dimensional
$D_{4}$-module. Consequently, $x_{21}$ is contained in a subgroup
$SL_{3}(q)$ and so $x_{21}$ has involution width at most 4 by Theorem
\ref{thm:sl 4 ref}. Finally $x_{40}$ is contained in a subgroup
$D_{4}(q)$. This $D_{4}(q)$ is the orthogonal group in characteristic
2 and the result follows from Theorem \ref{thm: orthog width}.

\textbf{Case 2: }\textit{Non-unipotent elements:} Assuming now that
$s$ is non trivial, it follows that $C_{G}(s)$ is a subsystem subgroup
of order at least $q^{26}$. Inspection of such subgroups (see \cite{Derizoitis untwisted e6})
shows that $C_{G}(s)$ has a quasisimple normal subgroup 
\[
C=D_{5}(q),\,D_{4}^{\delta}(q),\,A_{4}(q)\,\text{or}\,A_{5}(q).
\]
 Firstly suppose $C=D_{5}(q)$. Orders of unipotent centralizers in
$C$ are given in (\cite{Liebeck and Seitz book}, Table 8.6a) and
are determined by the Jordan block structure of $u\in C$ on the natural
10-dimensional $C$-module. As $C_{G}(s)=C\circ(q-1)$, it follows
that $|C_{C}(u)|>q^{25}$ and possible block structures are
\[
q\,\text{odd:}\,u=(J_{3},J_{1}^{7}),\,(J_{2}^{2},J_{1}^{6})\,\text{or}\,(J_{1}^{10})
\]
\[
q\,\text{even:}\,u=(J_{2}^{2},J_{1}^{6})\,\mbox{(2 classes)}\text{\, or}\,(J_{1}^{10}).
\]
 In all cases, $u$ is centralised by a fundamental subgroup $A_{1}(q)$
of $C$ and hence $u\in C_{G}(A_{1}(q))=A_{5}(q)$. As $s\in Z(C)$,
it follows that $g=su\in A_{5}(q)$ and the result follows from Theorem
\ref{thm:sl 4 ref}. Similarly, if $C=D_{4}^{\delta}(q)$, then we
can check possible centralizer dimensions in (\cite{Liebeck and Seitz book},
8.6) and the restriction $|C_{C}(u)|>q^{24}$ forces $u$ to be the
identity in $C$. Hence we can take $u$ to lie in a subgroup $H=A_{2}(q)$
of $C$ and then $g=su\in C_{G}(H)H=A_{2}(q)^{3}$ and we have the
result by Theorem \ref{thm:sl 4 ref}. The remaining cases where $C=A_{5}(q)$
or $A_{4}(q)$ can be dealt with in a similar manner. Unipotent centralizers
in linear groups are well known (see for example \cite{Liebeck and Seitz book})
and by checking orders it follows that the projection $u_{0}$ of
$u$ in $C$ must be $1$ or a transvection (i.e. $(J_{2},J_{1}^{4})$
or $(J_{2},J_{1}^{3})$ respectively). In either case we can again
embed $u$ in the subgroup $H=A_{2}(q)$ of $C$ and we have the result. 
\end{proof}

We now turn to the twisted case and complete the proof of Theorem
\ref{thm:E6 involution width result} when $G^{'}=\,^{2}E_{6}(q)$.
Recall that $\boldsymbol{G}^{F}=G=\,^{2}E_{6}(q)_{ad}$ and $G^{'}=\,^{2}E_{6}(q)$
with $q>2$. In the following, $\phi_{8,3'}$ and $\phi_{8,9''}$
denote unipotent characters of $G^{'}$ defined in (\cite{Carter rep theory},
Sec. 13.9).
\begin{lem}
\label{twisted E6 reduction to four characters in the structure constant sum}Suppose
$\chi\in Irr(G')$ such that $\chi(x)\chi(y)\ne0$. Then $\chi$ is
unipotent and $\chi\in\{1,\,St,\,\phi_{8,3'},\,\phi_{8,9''}\}$.
\end{lem}
\begin{proof}
Here the proof follows exactly as for Lemma \ref{E6 reduction to four characters in the structure constant sum}
so we shall be brief. If $\chi$ is unipotent then its degree is given
in (\cite{Carter rep theory}, 13.9) and we check that unless $\chi\in\{1,\,St,\,\phi_{8,3'},\,\phi_{8,9''}\}$,
it is defect zero for $r$ or $s$. Therefore if $\chi$ is not one
of these exceptions it will vanish on $x$ or $y$ respectively. All
non-unipotent characters will similarly vanish as for all semisimple
$1\neq t\in G^{*}$, $|C_{G^{*}}(t)|_{l}=1$ for either $l=r$ or
$l=s$. 
\end{proof}
\begin{lem}
Suppose that $g\in G$ such that $|C_{G'}(g)|\leq q^{28}.$ Then $g\in x^{G}\cdot y^{G}$.
\end{lem}
\begin{proof}
The proof follows exactly as in Lemma \ref{prop:dim less that 26 then result},
with an application of Lemma \ref{twisted E6 reduction to four characters in the structure constant sum}.
We omit the details but note the character degrees 
\[
St(1)=q^{36},\ \phi_{8,3'}=\frac{q^{3}}{2}\Phi_{2}^{4}\Phi_{6}^{2}\Phi_{10}\Phi_{18},\ \phi_{8,9"}=\frac{q^{15}}{2}\Phi_{2}^{4}\Phi_{6}^{2}\Phi_{10}\Phi_{18}.
\]
\end{proof}
Hence if $|C_{G'}(g)|\leq q^{28}$, then $g$ has involution width
at most $4$ by Lemma \ref{lem:x and y chosen in E6 are strongly real}.
Semisimple elements also have involution width at most 4 by Theorem
\ref{thm:(Gow-)Theorem for ss elements}. The following result handles
the remaining cases.
\begin{lem}
Let $g\in G'$ such that $|C_{G'}(g)|>q^{28}$. Then $g$ has involution
width at most 4.
\end{lem}
\begin{proof}
We employ the same method as Lemma \ref{prop:e6 width proof when the centraliser is too big},
embedding $g=su$ in a subgroup of $^{2}E_{6}(q)$ for which an involution
width result already exists. As before, we can assume that $u\neq1$. 

\textbf{Case 1: }\textit{Unipotent elements}: The structure of unipotent
centralizers are known in $G'$ (\cite{Liebeck and Seitz book} Table
22.1.3). Therefore the restriction $|C_{G}(u)|>q^{28}$ yields that
$g=u$ has one of the following labels in the algebraic group $\boldsymbol{G}$:
\[
A_{1},A_{1}^{2},A_{2},A_{1}^{3},A_{2}A_{1},A_{2}A_{1}^{2},A_{2}^{2}.
\]
Following the same argument as in the untwisted case, we see that
excluding type $A_{2}$, $u^{\boldsymbol{G}}\cap G$ is a single unipotent
class of $G$. In particular we note that in $\boldsymbol{G}$, the
labels above are all contained in the subsystem $A_{2}^{3}$. Thus
$u$ is contained in a corresponding finite subsystem subgroup of
$G'$. These are found in (\cite{Derizotis E type exceptional groups},
Sec.A(ii)) and the possibilities are $A_{2}(q)A_{2}(q^{2})$, $A_{2}^{-}(q^{3})$
or $A_{2}^{-}(q)^{3}$. Thus as these are simply central products
of copies of $SL_{3}$ and $SU_{3}$, the result follows by Theorems
\ref{thm:sl 4 ref} and \ref{lem:SU result for 3}. When $u$ has
label $A_{2}$, $u^{\boldsymbol{G}}\cap G$ contains two $G$-classes.
Here the reductive part of $C_{\boldsymbol{G}}(u)$ is $D=A_{2}^{2}$
with $D^{F}=A_{2}^{-}(q)^{2}$ or $A_{2}(q^{2})$. This has $\boldsymbol{G}$-centralizer
$\boldsymbol{L}'=A_{2}$ so $u\in\boldsymbol{L}'^{F}=A_{2}(q)$ and
the result follows from Lemma \ref{lem:SU result for 3}.

\textbf{Case 2: }\textit{Non-unipotent elements}: Assuming now that
$g=su$ where $s$ is nontrivial, it follows that $C_{G}(s)$ is a
subsystem subgroup of order at least $q^{28}$. Such centralizers
have been classified and checking against \cite{Deriziotis and liebeck twisted e6}
we see that $C_{G}(s)$ has a quasisimple normal subgroup 
\[
C=D_{5}^{-}(q),\,D_{4}^{\delta}(q),\,A_{4}^{-}(q)\,\text{or}\,A_{5}^{-}(q).
\]
Now we complete the proof just as in Lemma \ref{prop:e6 width proof when the centraliser is too big}.
\end{proof}
This concludes the proof of Theorem \ref{thm:E6 involution width result}
for the groups $E_{6}^{\epsilon}(q)$ and hence the proof of Theorem~\ref{mainthm}
is complete.

\end{document}